\documentclass{amsart}

\usepackage[usenames,dvipsnames]{color}
\usepackage{anysize}
\marginsize{3cm}{3cm}{4cm}{4cm}
\newtheorem{thm}{Theorem}[section]
\newtheorem{cor}[thm]{Corollary}
\newtheorem{lem}[thm]{Lemma}
\newtheorem{prop}[thm]{Proposition}
\theoremstyle{definition}
\newtheorem{defn}[thm]{Definition}
\theoremstyle{remark}
\newtheorem{rem}[thm]{Remark}
\numberwithin{equation}{section}

\newcommand{\prode}[1]{\left\langle#1\right\rangle}

\newcommand{\set}[1]{\left\{#1\right\}}

\newcommand{\paren}[1]{\left(#1\right)}

\begin{document}

\title[Oblique dual fusion frames]{Oblique dual fusion frames}

\author[Sigrid B. Heineken]{Sigrid B. Heineken$^{1,*}$}%
%{Sigrid B. Heineken\\Departamento de Matem\'atica \\Facultad de Ciencias Exactas y Naturales\\
%Universidad de Buenos Aires\\ Pabell\'on I\\ Pabell\'on I\\
%C1428EGA C.A.B.A.\\ Argentina\\ IMAS, UBA-CONICET\\Argentina.}
%\email{sheinek@dm.uba.ar}%

\author[Patricia M. Morillas]{Patricia M. Morillas$^2$\\\\ $^{1}$\textit{D\lowercase{epartamento de} M\lowercase{atem\'atica}, FCE\lowercase{y}N, U\lowercase{niversidad de }B\lowercase{uenos} A\lowercase{ires}, P\lowercase{abell\'on} I, C\lowercase{iudad }U\lowercase{niversitaria}, IMAS, UBA-CONICET, C1428EGA C.A.B.A., A\lowercase{rgentina}\\ $^2$ I\lowercase{nstituto de }M\lowercase{atem\'{a}tica }A\lowercase{plicada} S\lowercase{an} L\lowercase{uis, }UNSL-CONICET \lowercase{and} D\lowercase{epartamento de} M\lowercase{atem\'{a}tica}, FCFM\lowercase{y}N, UNSL, E\lowercase{j\'{e}rcito de los }A\lowercase{ndes 950, 5700 }S\lowercase{an} L\lowercase{uis,} A\lowercase{rgentina}}}
%\address{Patricia M. Morillas\\Instituto de Matem\'{a}tica Aplicada San Luis\\
%         UNSL-CONICET\\
%         and Departamento de Matem\'{a}tica\\
%         FCFMyN, UNSL\\
%         Ej\'{e}rcito de los Andes 950, 5700 San Luis, Argentina}%
%\email{morillas@unsl.edu.ar}%
%

%\date{}%
%\dedicatory{}%
%\commby{}%

\thanks{* Corresponding author.\\
\textit{E-mail addresses:}  sheinek@dm.uba.ar (S. B. Heineken),
morillas@unsl.edu.ar (P. M. Morillas)}
% ---------------------------------------------------------------

\begin{abstract}
We introduce and develop the concept of oblique duality for fusion
frames. This concept provides a mathematical framework to deal with
problems in distributed signal processing where the signals,
considered as elements in a Hilbert space and under certain
consistency requirements, are analyzed in one subspace and are
reconstructed in another subspace.

\bigskip

\bigskip

{\bf Key words:} Frames, Oblique dual frames, Fusion frames, Oblique
dual fusion frames, Consistent reconstruction, Oblique projections.

\medskip

{\bf AMS subject classification:} Primary 42C15; Secondary 42C40,
46C05, 47B10.

\end{abstract}

\maketitle

% ----------------------------------------------------------------
\section{Introduction}

Fusion frames  \cite{Casazza-Kutyniok (2004), Casazza-Kutyniok-Li
(2008)} (see also \cite[Chapter 13]{Casazza-Kutyniok (2012)})
generalize the notion of frames \cite{{Casazza-Kutyniok (2012)},
Christensen (2016), Kovacevic-Chebira (2008)}. They are suitable in
applications such as signal processing and sampling theory, in
situations where one has to implement a local combination of data
vectors. They allow representations of the elements of a separable
Hilbert space using packets of linear coefficients.

Oblique dual frames have been introduced in \cite{Eldar (2003a)} and
studied in \cite{Eldar (2003b), Eldar-Werther
(2005),Christensen-Eldar (2004), Christensen-Eldar (2006)}. Oblique
dual frames are useful in cases in which the analysis of a signal
and its reconstruction have to be done in different subspaces.

The aim of this paper is to introduce and develop the concept of
oblique duality for fusion frames. This concept arises naturally
from the notion of Eldar of oblique dual frames in \cite{Eldar
(2003a)} and our definition of dual fusion frames in \cite
{Heineken-Morillas-Benavente-Zakowicz (2014), Heineken-Morillas
(2014)}.

The paper is organized as follows. In Section 2 we give an overview
of oblique projections, left inverses, frames, fusion frames and
fusion frame systems.

We begin Section 3 introducing the concept of consistent
reconstruction for fusion frames as a motivation of oblique duality.
Then we introduce the definitions of oblique dual fusion frames and
oblique dual fusion frame systems together with its basic
properties. We present two classes of oblique dual fusion frames of
special interest: the block-diagonal and, a subclass of them, the
component preserving ones. The advantage of block diagonals is that
they lead to a reconstruction formula that has a simpler expression.
Oblique dual fusion frame systems are an example of this type.
Component preserving duals are those that behave more similar to
classical vectorial oblique duals, in particular, in the way that
they can be obtained.

In Section 4 we analyze how the concepts of block-diagonal oblique
dual frame, oblique dual fusion frame system and oblique dual frame
are connected.

Section 5 presents statements that permit to obtain oblique dual
fusion frames from dual fusion frames and viceversa.

Section 6 includes several results that give methods for
constructing oblique dual fusion frames and oblique dual fusion
frame systems. They provide different alternatives to select the
most suitable from the computational point of view.

Finally, in Section 7 we introduce the concept of canonical oblique
dual fusion frame. Some basic properties and a characterization are
presented. Then we study when the canonical oblique dual is the
unique dual and the existence of non-canonical oblique duals.

% ----------------------------------------------------------------
\section{Preliminaries}

We begin introducing some notation and then we briefly review
definitions and properties that we use later.

\subsection{Notation}

We consider $\mathcal{H}, \mathcal{K}$ separable Hilbert spaces over
$\mathbb{F}=\mathbb{R}$ or $\mathbb{F}=\mathbb{C}$. The space of
bounded operators from $\mathcal{H}$ to $\mathcal{K}$ will be
denoted by $L(\mathcal{H},\mathcal{K})$ (we write $L(\mathcal{H})$
for $L(\mathcal{H},\mathcal{H})$). For $T \in
L(\mathcal{H},\mathcal{K})$ we denote the image, the null space and
the adjoint of $T$ by  ${\rm Im}(T)$, ${\rm Ker}(T)$ and $T^{*}$,
respectively. If $T$ has closed range we also consider the
Moore-Penrose pseudo-inverse of $T$ denoted by $T^{\dagger}$. The
inner product and the norm in $\mathcal{H}$ will be denoted by
$\langle \cdot ,\cdot \rangle$ and $\|\cdot\|$, respectively. If $T
\in L(\mathcal{H},\mathcal{K})$, then $\|T\|_{F}$ and $\|T\|_{sp}$
denote the Frobenius and the spectral norms of $T$, respectively.

Let $I$ be a countable index set. If $\{\mathcal{H}_{i}\}_{i \in I}$
is a sequence of Hilbert spaces, we consider the Hilbert space

\centerline{$\oplus_{i \in I}\mathcal{H}_{i} = \{(f_{i})_{i \in
I}:f_{i} \in \mathcal{H}_{i} \text{ and } \{\|f_{i}\|\}_{i \in I}
\in \ell^{2}(I)\}$}

\noindent with inner product $\langle(f_{i})_{i \in I},(g_{i})_{i
\in I}\rangle=\sum_{i \in I}\langle f_{i}, g_{i}\rangle.$

For $J \subseteq I$ let $\chi_{J} : I \rightarrow \{0, 1\}$ be the
characteristic function of $J.$ We abbreviate
$\chi_{\{j\}}=\chi_{j}$.

\subsection{Oblique projections and left inverses}

In the sequel $\mathcal{V}$ and $\mathcal{W}$ will be two closed
subspaces of $\mathcal{H}$ such that $\mathcal{H}=\mathcal{V} \oplus
\mathcal{W}^{\perp} $. By \cite[Lemma 2.1]{Christensen-Eldar (2004)}
this is equivalent to $\mathcal{H}=\mathcal{W} \oplus
\mathcal{V}^{\perp}$.

The oblique projection onto $\mathcal{V}$ along
$\mathcal{W}^{\perp}$, is the unique operator that satisfies

\centerline{$\pi_{\mathcal{V},\mathcal{W}^{\perp}}f=f \text{ for all
} f \in \mathcal{V}$,}

\centerline{ \, $\pi_{\mathcal{V},\mathcal{W}^{\perp}}f=0 \text{ for
all
 } f \in \mathcal{W}^{\perp}$.}

Equivalently, ${\rm
Im}(\pi_{\mathcal{V},\mathcal{W}^{\perp}})=\mathcal{V}$ and ${\rm
Ker}(\pi_{\mathcal{V},\mathcal{W}^{\perp}})=\mathcal{W}^{\perp}$. If
$\mathcal{W}=\mathcal{V}$ we obtain the orthogonal projection onto
$\mathcal{W}$, which we denote by $\pi_{\mathcal{W}}$. The next
properties involving oblique projections will be used throughout the
paper.

\begin{lem}\label{L composicion proyeccion oblicua con ortogonal}
Let $\mathcal{V}$ and $\mathcal{W}$ be two closed subspaces of
$\mathcal{H}$ such that $\mathcal{H}=\mathcal{V} \oplus
\mathcal{W}^{\perp}$. Then
\begin{enumerate}

  \item $\pi_{\mathcal{V},\mathcal{W}^{\perp}}\pi_{\mathcal{W}}=\pi_{\mathcal{V},\mathcal{W}^{\perp}}$

  \item $\pi_{\mathcal{W}}\pi_{\mathcal{V},\mathcal{W}^{\perp}}=\pi_{\mathcal{W}}$.

\end{enumerate}
\end{lem}
\begin{proof}

(1) $\pi_{\mathcal{V},\mathcal{W}^{\perp}}\pi_{\mathcal{W}}=
\pi_{\mathcal{V},\mathcal{W}^{\perp}}(\pi_{\mathcal{W}}+\pi_{\mathcal{W}^{\perp}})
=  \pi_{\mathcal{V},\mathcal{W}^{\perp}}$.

(2)
$\pi_{\mathcal{W}}\pi_{\mathcal{V},\mathcal{W}^{\perp}}=\pi_{\mathcal{W}}(\pi_{\mathcal{V},\mathcal{W}^{\perp}}+\pi_{\mathcal{W}^{\perp},\mathcal{V}})=\pi_{\mathcal{W}}$.
\end{proof}

Let $\mathcal{V}$ and $\mathcal{W}$ be two closed subspaces of
$\mathcal{H}$ such that $\mathcal{H}=\mathcal{V} \oplus
\mathcal{W}^{\perp}$. If $T \in L(\mathcal{H},\mathcal{K})$ and
${\rm Ker}(T)=\mathcal{W}^{\perp}$, we denote by
$\mathfrak{L}_{T}^{\mathcal{V},\mathcal{W}^{\perp}}$ the set of
oblique left inverses of $T$ on $\mathcal{V}$ along
$\mathcal{W}^{\perp}$ which image is equal to $\mathcal{V}$, i. e.,

\centerline{$\mathfrak{L}_{T}^{\mathcal{V},\mathcal{W}^{\perp}}=\{U
\in
L(\mathcal{K},\mathcal{H}):UT=\pi_{\mathcal{V},\mathcal{W}^{\perp}}
\text{ and } {\rm Im}(U)=\mathcal{V}\}$.}

The following lemma is useful to obtain oblique left inverses.

\begin{lem}\label{L T inyectiva en V}
Let $\mathcal{V}$ and $\mathcal{W}$ be two closed subspaces of
$\mathcal{H}$ such that $\mathcal{H}=\mathcal{V} \oplus
\mathcal{W}^{\perp}$. Let $T \in L(\mathcal{H},\mathcal{K})$ such
that ${\rm Ker}(T)=\mathcal{W}^{\perp}$. Then $T_{|\mathcal{V}}$ is
injective.
\end{lem}

\begin{proof}
Let $f, g \in \mathcal{V}$ such that $Tf=Tg$. Since ${\rm
Ker}(T)=\mathcal{W}^{\perp}$ the last equality is equivalent to
$T\pi_{\mathcal{W}}f=T\pi_{\mathcal{W}}g$. Since $T_{|\mathcal{W}}$
is injective, $\pi_{\mathcal{W}}f=\pi_{\mathcal{W}}g$. Therefore,
$f-\pi_{\mathcal{W}^{\perp}}f=g-\pi_{\mathcal{W}^{\perp}}g$. Since
$\mathcal{H}=\mathcal{V} \oplus \mathcal{W}^{\perp}$, $f=g$.
\end{proof}

One way to get $U \in
\mathfrak{L}_{T}^{\mathcal{V},\mathcal{W}^{\perp}}$ is the
following: By Lemma~\ref{L T inyectiva en V}, if $g \in
T(\mathcal{V})$ there exists a unique $f \in \mathcal{V}$ such that
$Tf=g$. Set $Ug = f$. If $g \notin T(\mathcal{V})$ there are several
possibilities, for example, $Ug=U(g_{1}+g_{2})=Ug_{1}$ with $g_{1}
\in T(\mathcal{V})$ and $g_{2} \in T(\mathcal{V})^{\perp}$.

If $\mathcal{V}=\mathcal{W}$, we write
$\mathfrak{L}_{T}^{\mathcal{W}}=\{U\in
L(\mathcal{K},\mathcal{H}):UT=\pi_{\mathcal{W}} \text{ and } {\rm
Im}(U)=\mathcal{W}\}$. This is the set of left inverses of $T$ on
$\mathcal{W}$ such that ${\rm Im}(U)=\mathcal{W}$. The next
proposition relates the sets $\mathfrak{L}_{T}^{\mathcal{W}}$ and
$\mathfrak{L}_{T}^{\mathcal{V},\mathcal{W}^{\perp}}$.

\begin{prop}\label{P relacion inversa inversa oblicua}
Let $\mathcal{V}$ and $\mathcal{W}$ be two closed subspaces of
$\mathcal{H}$ such that $\mathcal{H}=\mathcal{V} \oplus
\mathcal{W}^{\perp}$. Let $T \in L(\mathcal{H},\mathcal{K})$ such
that ${\rm Ker}(T)=\mathcal{W}^{\perp}$. The map $A \in
\mathfrak{L}_{T}^{\mathcal{W}} \mapsto
\pi_{\mathcal{V},\mathcal{W}^{\perp}}A \in
\mathfrak{L}_{T}^{\mathcal{V},\mathcal{W}^{\perp}}$ is a linear
bijection, and its inverse is the map $B \in
\mathfrak{L}_{T}^{\mathcal{V},\mathcal{W}^{\perp}} \mapsto
\pi_{\mathcal{W}}B \in \mathfrak{L}_{T}^{\mathcal{W}}$.
\end{prop}

\begin{proof}
First we will show that the map $A \in
\mathfrak{L}_{T}^{\mathcal{W}} \mapsto
\pi_{\mathcal{V},\mathcal{W}^{\perp}}A \in
\mathfrak{L}_{T}^{\mathcal{V},\mathcal{W}^{\perp}}$ is well defined.
Let $A\in\mathfrak{L}_{T}^{\mathcal{W}}$. We have,
$\pi_{\mathcal{V},\mathcal{W}^{\perp}}AT=\pi_{\mathcal{V},\mathcal{W}^{\perp}}\pi_{\mathcal{W}}=\pi_{\mathcal{V},\mathcal{W}^{\perp}}$.
On the other hand, ${\rm
Im}(\pi_{\mathcal{V},\mathcal{W}^{\perp}}A)\subseteq\mathcal{V}$.
Let $f\in\mathcal{V}$. Since ${\rm Im}(A)=\mathcal{W}$ there exists
$g\in \mathcal{K}$ such that $\pi_{\mathcal{W}}f=Ag$ and then
$f=\pi_{\mathcal{V},\mathcal{W}^{\perp}}f=\pi_{\mathcal{V},\mathcal{W}^{\perp}}\pi_{\mathcal{W}}f=\pi_{\mathcal{V},\mathcal{W}^{\perp}}Ag$.
Therefore, ${\rm
Im}(\pi_{\mathcal{V},\mathcal{W}^{\perp}}A)=\mathcal{V}$. This shows
that $\pi_{\mathcal{V},\mathcal{W}^{\perp}}A\in\mathfrak{L}_
{T}^{\mathcal{V},\mathcal{W}^{\perp}}$.

The linearity is trivial. Now we will prove that it is a bijection
showing that the map $B \in
\mathfrak{L}_{T}^{\mathcal{V},\mathcal{W}^{\perp}} \mapsto
\pi_{\mathcal{W}}B \in \mathfrak{L}_{T}^{\mathcal{W}}$ is its
inverse. First we note that in a similar manner as before it can be
proved that it is a well defined linear map. Let
$A\in\mathfrak{L}_{T}^{\mathcal{W}}$. Since ${\rm
Im}(A)=\mathcal{W}$,
$\pi_{\mathcal{W}}\pi_{\mathcal{V},\mathcal{W}^{\perp}}A=\pi_{\mathcal{W}}A=A$.
Let now $B \in \mathfrak{L}_{T}^{\mathcal{V},\mathcal{W}^{\perp}}$.
Using that ${\rm Im}(B)=\mathcal{V}$,
$\pi_{\mathcal{V},\mathcal{W}^{\perp}}\pi_{\mathcal{W}}B=\pi_{\mathcal{V},\mathcal{W}^{\perp}}B=B$.
This proves that each map is the inverse of the other.
\end{proof}

\subsection{Frames}

The concept of frame has been introduced by Duffin  and Schaeffer in
\cite{Duffin-Schaeffer (1952)}. Using a frame, each element of a
Hilbert space has a representation which in general is not unique.
This flexibility makes them attractive for many applications
involving signal expansions.

We will now recall the concept of frame for a closed subspace of
$\mathcal{H}$.

\begin{defn}\label{D frame}
Let $\mathcal{W}$ be a closed subspace of $\mathcal{H}$ and
$\{f_{i}\}_{i \in I} \subset \mathcal{W}$. Then $\{f_{i}\}_{i \in
I}$ is a \emph{frame} for $\mathcal{W}$, if there exist constants $0
< \alpha \leq \beta < \infty$ such that
\begin{equation}\label{E cond frame}
\alpha\|f\|^{2} \leq \sum_{i \in I}|\langle f,f_{i}\rangle |^{2}
\leq \beta\|f\|^{2}  \text{ for all $f \in \mathcal{W}$.}
\end{equation}
\end{defn}

If the right inequality in (\ref{E cond frame}) is satisfied,
$\{f_{i}\}_{i \in I}$ is a {\it Bessel sequence} for $\mathcal{W}$.
The constants $\alpha$ and $\beta$ are the {\it frame bounds}. In
case $\alpha=\beta,$ we call $\{f_{i}\}_{i \in I}$ an {\it
$\alpha$-tight frame}, and if $\alpha=\beta=1$ it is a {\it Parseval
frame} for $\mathcal{W}$.

To a Bessel sequence $\mathcal{F}=\{f_{i}\}_{i \in I}$ for
$\mathcal{W}$ we associate the {\it synthesis operator}

\centerline{$T_{\mathcal{F}}:\ell^2(I)\rightarrow \mathcal{H},$
$T_{\mathcal{F}}\{c_i\}_{i\in I}=\sum_{i\in I}c_if_i,$}

\noindent the {\it analysis operator}

\centerline{$T_{\mathcal{F}}^{*}: \mathcal{H}\rightarrow \ell^2(I)$,
$T_{\mathcal{F}}^{*}f=\{\langle f,f_i\rangle \}_{i\in I},$}

\noindent and the {\it frame operator}

\centerline{$S_{\mathcal{F}}=T_{\mathcal{F}}T_{\mathcal{F}}^{*}$.}

\noindent A Bessel sequence $\mathcal{F}=\{f_{i}\}_{i \in I}$ for
$\mathcal{W}$ is a frame for $\mathcal{W}$ if and only ${\rm
Im}(T_{\mathcal{F}})=\mathcal{W}$, or equivalently,
$S_{\mathcal{F}}$ is invertible when restricted to $\mathcal{W}$.
Furthermore, $\mathcal{F}$ is an $\alpha$-tight frame for
$\mathcal{W}$ if and only if $S_{\mathcal{F}}=\alpha
\pi_{\mathcal{W}}$.

If the subspace $\mathcal{W}$ is finite-dimensional we will consider
finite frames for it, i.e., frames with a finite number of elements.
It is worth to mention that if $\text{dim}(\mathcal{W})<\infty$ then
$\{f_{i}\}_{i \in I} \subset \mathcal{H}$ is a frame for
$\mathcal{W}$ if and only $\text{span}\{f_{i}\}_{i \in
I}=\mathcal{W}$.

For more details about frames we refer the reader to
\cite{Casazza-Kutyniok (2012), Christensen (2016), Kovacevic-Chebira
(2008)}. The concept of oblique dual frame \cite{Eldar (2003a),
Eldar (2003b), Eldar-Werther (2005)} is defined as follows:

\begin{defn}\label{D oblique dual frame}
Let $\mathcal{W}$ and $\mathcal{V}$ be two closed subspaces of
$\mathcal{H}$ such that $\mathcal{H}=\mathcal{V}\oplus
\mathcal{W}^{\perp}$. Let $\mathcal{F}=\{f_{i}\}_{i \in I}$ be a
frame for $\mathcal{W}$ and $\mathcal{G}=\{g_{i}\}_{i \in I}$ be a
frame for $\mathcal{V}$. If

\centerline{$T_{\mathcal{G}}T_{\mathcal{F}}^{*}=\pi_{\mathcal{V},\mathcal{W}^{\perp}},$}

\noindent we say that $\mathcal{G}$ is an {\em oblique dual frame}
of $\mathcal{F}$ on $\mathcal{V}$.
\end{defn}

The sequence
$\{\pi_{\mathcal{V},\mathcal{W}^{\perp}}S_{\mathcal{F}}^{\dagger}f_{i}\}_{i
\in I}$ is the \emph{canonical oblique dual frame} of $\{f_{i}\}_{i
\in I}$ on $\mathcal{V}$.

\begin{rem}
When $\mathcal{V}=\mathcal{W}$ we obtain the classical duals and we
simply say dual frame instead of oblique dual frame on
$\mathcal{W}$.
\end{rem}

A {\it Riesz basis} for $\mathcal{W}$ is a frame for $\mathcal{W}$
which is also a basis. Observe that a Riesz basis has a unique dual,
the canonical one.

\subsection{Fusion frames}

Fusion frames were introduced by Casazza and Kutyniok in
\cite{Casazza-Kutyniok (2004)} under the name of {\it frames of
subspaces}. They turned out to be a useful tool for handling
problems in sensor networking, distributed processing, etc.
Throughout the paper we will work with fusion frames for closed
subspaces of $\mathcal{H}$.

Assume $\{W_{i}\}_{i \in I}$ is a family of closed subspaces in
$\mathcal{W}$, and $\{w_{i}\}_{i \in I}$ a family of weights, i.e.,
$w_{i}
> 0$ for all $i \in I$. We denote $\{W_{i}\}_{i \in I}$ with $\mathbf{W}$, $\{w_{i}\}_{i
\in I}$ with $\mathbf{w}$ and $\{(W_i,w_{i})\}_{i \in I}$ with
$(\mathbf{W},\mathbf{w})$. If $T \in L(\mathcal{H},\mathcal{K})$ we
write $(T\mathbf{W},\mathbf{w})$ for $\{(TW_i, w_{i})\}_{i \in I}.$

We consider the Hilbert space $\mathcal{K}_{\mathcal{W}}:=\oplus_{i
\in I}W_{i}$.

\begin{defn}\label{D fusion frame}
We say that $(\mathbf{W},\mathbf{w})$ is a {\it fusion frame} for
$\mathcal{W}$, if there exist constants $0 < \alpha \leq \beta <
\infty$ such that
\begin{equation}\label{E cond ff}
\alpha\|f\|^{2} \leq \sum_{i \in I}w_{i}^{2}\|\pi_{W_{i}}(f)\|^{2}
\leq \beta\|f\|^{2}  \text{ for all $f \in \mathcal{W}$.}
\end{equation}
\end{defn}

We call $\alpha$ and $\beta$ the \textit{fusion frame bounds}. The
family $(\mathbf{W},\mathbf{w})$ is called an $\alpha$-\textit{tight
fusion frame} for $\mathcal{W}$, if in (\ref{E cond ff}) the
constants $\alpha$ and $\beta$ can be chosen so that $\alpha =
\beta$, and a \textit{Parseval fusion frame} for $\mathcal{W}$
provided that $\alpha = \beta = 1.$ If $(\mathbf{W},\mathbf{w})$ has
an upper fusion frame bound, but not necessarily a lower bound, it
is called a \textit{Bessel fusion sequence} for $\mathcal{W}$ with
Bessel fusion bound $\beta.$ If $w_{i}=c$ for all $i\in I,$ we write
$\mathbf{w}=c$. If $\mathcal{W}$ is the direct sum of the $W_i$ we
say that $(\mathbf{W},{\bf w})$ is a \textit{Riesz fusion basis} for
$\mathcal{W}$. We will refer to a fusion frame that is not a Riesz
fusion basis as an overcomplete fusion frame. A fusion frame
$(\mathbf{W},1)$ is an \textit{orthonormal fusion basis} for
$\mathcal{W}$ if {$\mathcal{W}$ is the orthogonal sum of the
subspaces $W_{i}.$}

To a Bessel fusion sequence $(\mathbf{W},\mathbf{w})$ for
$\mathcal{W}$ we  associate the \textit{synthesis operator}

\centerline{$T_{\mathbf{W},\mathbf{w}} : \mathcal{K}_{\mathcal{W}}
\rightarrow  \mathcal{H},$
$\,\,\,T_{\mathbf{W},\mathbf{w}}(f_{i})_{i \in I}=\sum_{i \in
  I}w_{i}f_{i},$}

\noindent the \textit{analysis operator}

\centerline{$T_{\mathbf{W},\mathbf{w}}^{*} : \mathcal{H} \rightarrow
\mathcal{K}_{\mathcal{W}},$
$\,\,\,T_{\mathbf{W},\mathbf{w}}^{*}f=(w_{i}\pi_{W_{i}}(f))_{i \in
I}$}

\noindent and the {\it fusion frame operator}

\centerline{$S_{\mathbf{W},\mathbf{w}}=T_{\mathbf{W},\mathbf{w}}T_{\mathbf{W},\mathbf{w}}^{*}.$}

As it happens for frames, $(\mathbf{W},\mathbf{w})$ is a Bessel
fusion sequence for $\mathcal{W}$ if and only if
$T_{\mathbf{W},\mathbf{w}}$ is a well defined bounded linear
operator. A Bessel fusion sequence $(\mathbf{W},\mathbf{w})$  for
$\mathcal{W}$ is a fusion frame for $\mathcal{W}$ if and only if
${\rm Im}(T_{\mathbf{W},\mathbf{w}})=\mathcal{W}$, or equivalently,
$S_{\mathbf{W},\mathbf{w}}$ restricted to $\mathcal{W}$ is
bijective. Additionally, $(\mathbf{W},\mathbf{w})$ is an
$\alpha$-tight fusion frame for $\mathcal{W}$ if and only if
$S_{\mathbf{W},\mathbf{w}}=\alpha \pi_{\mathcal{W}}.$

For finite-dimensional subspaces $\mathcal{W}$ we will consider
finite fusion frames, i.e., fusion frames with a finite set of
indices. Note that if $\text{dim}(\mathcal{W})<\infty$ then
$(\mathbf{W},\mathbf{w})$ is a frame for $\mathcal{W}$ if and only
$\text{span}\cup_{i \in I}W_{i}=\mathcal{W}$.

Having fusion frames allows local processing in each of the
subspaces. In view of this, having a set of local frames for its
subspaces is convenient.

\begin{defn}
Let $\mathcal{W}$ be a closed subspace of $\mathcal{H}$, let
$(\mathbf{W},{\bf w})$ be a fusion frame (\emph{Bessel fusion
sequence}) for $\mathcal{W}$, and let $\{f_{i,l}\}_{l\in L_i}$ be a
frame for $W_{i}$ for $i \in I$. Then $\{(W_i,w_i,\{f_{i,l}\}_{l\in
L_i})\}_{i \in I}$ is called a \emph{fusion frame system}
(\emph{Bessel fusion system}) for $\mathcal{W}$.
\end{defn}

Throughout this work we will use the notation
$\mathcal{F}_{i}=\{f_{i,l}\}_{l \in L_i}$,
$\mathcal{F}=\{\mathcal{F}_{i}\}_{i \in I},$ ${\bf
w}\mathcal{F}=\{w_{i}\mathcal{F}_{i}\}_{i \in I},$ and we write
$(\mathbf{W}, {\bf w}, \mathcal{F})$ for
$\{(W_i,w_i,\{f_{i,l}\}_{l\in L_i})\}_{i \in I}$. If $T \in
L(\mathcal{H},\mathcal{K})$ we write $T\mathcal{F}$ for
$\{\{T{f}_{i,l}\}_{l\in L_i}\}_{i \in I}$ and $T\mathcal{F}_{i}$ for
$\{Tf_{i,l}\}_{l \in L_i}$.

\begin{thm}\cite[Theorem 3.2]{Casazza-Kutyniok (2004)}\label{T wF marco sii WwF fusion frame system}
Let $\mathcal{W}$ be a closed subspace of $\mathcal{H}$. Given
$(\mathbf{W},\mathbf{w})$, let $\mathcal{F}_{i}$ be a frame for
$W_{i}$ with frame bounds $\alpha_{i}, \beta_{i}$ for each $i \in I$
such that $0 < \alpha ={\rm inf}_{i \in I}\alpha_{i} \leq {\rm
sup}_{i \in I}\beta_{i}= \beta < \infty$. The following assertions
are equivalents:
\begin{enumerate}
  \item $\mathbf{w}\mathcal{F}$ is a frame for $\mathcal{W}$.
  \item $(\mathbf{W},\mathbf{w})$ is a fusion frame for
  $\mathcal{W}$.
\end{enumerate}
If $(\mathbf{W},\mathbf{w})$ is a fusion frame for $\mathcal{W}$
with fusion frame bounds $\gamma$ and $\delta$, then
$\mathbf{w}\mathcal{F}$ is a frame for $\mathcal{W}$ with frame
bounds $\alpha \gamma$ and $\beta \delta$. If
$\mathbf{w}\mathcal{F}$ is a frame for $\mathcal{W}$ with frame
bounds $\gamma$ and $\delta$, then $(\mathbf{W}, {\bf w})$ is a
fusion frame for $\mathcal{W}$ with fusion frame bounds
$\frac{\gamma}{\beta}$ and $\frac{\delta}{\alpha}$.

The previous assertions are valid if we replace fusion frame by
Bessel fusion sequence and consider only the upper bounds.
\end{thm}

For more details about fusion frames and fusion frame systems we
refer the reader to \cite{Casazza-Kutyniok (2004),
Casazza-Kutyniok-Li (2008)} (see also \cite[Chapter
13]{Casazza-Kutyniok (2012)}).

% ----------------------------------------------------------------
\section{Oblique duality for fusion frames and fusion frame systems}

One reason for considering oblique duality is the so called
consistent reconstruction. Based on the vectorial case \cite{Eldar
(2003a), Eldar-Werther (2005)} and the relation between frames and
fusion frames, we next introduce the concept of consistent
reconstruction for fusion frames.

\subsection{Consistent reconstruction}

Let $\mathcal{W}$ and $\mathcal{V}$ be two closed subspaces of
$\mathcal{H}$. Let $(\mathbf{W},{\bf w})$ be a fusion frame for
$\mathcal{W}$. Assume that the measurements $T_{\mathbf{W},{\bf
w}}^{*}f=(w_{i}\pi_{W_{i}}f)_{i \in I}$ of an unknown signal $f \in
\mathcal{H}$ are given. Our goal is the reconstruction of $f$ from
these measurements using a fusion frame $({\bf V},{\bf v})$ for
$\mathcal{V}$ in such a way that the reconstruction $\widehat{f}$ is
a good approximation of $f$. Specifically  the following two
conditions are required:
\begin{enumerate}
  \item[(i)] \textit{Uniqueness of the reconstructed signal}: If $f, g \in \mathcal{V}$ and $T_{\mathbf{W},{\bf
w}}^{*}f=T_{\mathbf{W},{\bf w}}^{*}g$, then $f=g$.

  \item[(ii)] \textit{Consistent sampling}: $T_{\mathbf{W},{\bf w}}^{*}\widehat{f}=T_{\mathbf{W},{\bf w}}^{*}f$
for all $f \in \mathcal{H}$.

\end{enumerate}

Requirement (i) is equivalent to
$\mathcal{V}\cap\mathcal{W}^{\perp}=\{0\}$. To see this, suppose
that (i) holds and consider $f \in
\mathcal{V}\cap\mathcal{W}^{\perp}$. Since $f \in
\mathcal{W}^{\perp}$, $T_{\mathbf{W},{\bf
w}}^{*}f=0=T_{\mathbf{W},{\bf w}}^{*}0$. Since $f \in \mathcal{V}$,
by (i) this implies that $f=0$, and thus
$\mathcal{V}\cap\mathcal{W}^{\perp}=\{0\}$. Suppose now that
$\mathcal{V}\cap\mathcal{W}^{\perp}=\{0\}$. Let $f, g \in
\mathcal{V}$ such that $T_{\mathbf{W},{\bf w}}^{*}(f-g)=0$. Thus
$f-g \in {\rm Ker}(T_{\mathbf{W},{\bf w}}^{*})={\rm
Im}(T_{\mathbf{W},{\bf w}})^{\perp}=\mathcal{W}^{\perp}$.
Consequenlty, $f-g=0$. Therefore, (i) holds.

In case that (ii) is satisfied we say that $\widehat{f} \in
\mathcal{V}$ is a \emph{consistent reconstruction} of $f \in
\mathcal{H}$ on $\mathcal{V}$ along $\mathcal{W}^{\perp}$.

From (i) and (ii), we deduce that if $f \in \mathcal{V}$ then
$\widehat{f}=f$. So in this case, $f$ can be perfectly
reconstructed.

The next result shows that consistent reconstruction is linked to
oblique projections.

\begin{thm}\label{T consistent reconstruction}
Let $\mathcal{W}$ and $\mathcal{V}$ be two closed subspaces of
$\mathcal{H}$ such that
$\mathcal{H}=\mathcal{V}\oplus\mathcal{W}^{\perp}$. Let
$(\mathbf{W},{\bf w})$ be a fusion frame for $\mathcal{W}$. Then
$\widehat{f} \in \mathcal{V}$ is a consistent reconstruction of $f
\in \mathcal{H}$ if and only if
$\widehat{f}=\pi_{\mathcal{V},\mathcal{W}^{\perp}}f$.
\end{thm}

\begin{proof}
Suppose that $\widehat{f} \in \mathcal{V}$ is a consistent
reconstruction of $f \in \mathcal{W}$, i.e. (ii) holds. Then
$\widehat{f}-f \in {\rm Ker}(T_{\mathbf{W},{\bf w}}^{*})={\rm
Im}(T_{\mathbf{W},{\bf w}})^{\perp}=\mathcal{W}^{\perp}$. Thus
$f=\widehat{f}+(f-\widehat{f})$ where $\widehat{f} \in \mathcal{V}$
and $\widehat{f}-f \in \mathcal{W}^{\perp}$. Taking into account
that $\mathcal{H}=\mathcal{V}\oplus\mathcal{W}^{\perp}$, this
implies that $\widehat{f}=\pi_{\mathcal{V},\mathcal{W}^{\perp}}f$
and $\widehat{f}-f=\pi_{\mathcal{W}^{\perp},\mathcal{V}}f$.

Assume now that
$\widehat{f}=\pi_{\mathcal{V},\mathcal{W}^{\perp}}f$. Since
$\mathcal{H}=\mathcal{V}\oplus\mathcal{W}^{\perp}$, $\widehat{f}-f
=\pi_{\mathcal{W}^{\perp},\mathcal{V}}f \in \mathcal{W}^{\perp} =
{\rm Im}(T_{\mathbf{W},{\bf w}})^{\perp}={\rm
Ker}(T_{\mathbf{W},{\bf w}}^{*})$. Therefore $T_{\mathbf{W},{\bf
w}}^{*}(\widehat{f}-f)=0$ and (ii) holds.
\end{proof}

\subsection{Oblique dual fusion frames}

In \cite{Heineken-Morillas-Benavente-Zakowicz (2014),
Heineken-Morillas (2014)} the concepts of dual fusion frame and dual
fusion frames system are introduced and studied. Motivated by these
concepts and by Definition~\ref{D oblique dual frame} we introduce
now the definition of oblique dual fusion frame and later the one of
oblique dual fusion frame system.

\begin{defn}\label{D oblique fusion frame dual}
Let $\mathcal{W}$ and $\mathcal{V}$ be two closed subspaces of
$\mathcal{H}$ such that $\mathcal{H}=\mathcal{V}\oplus
\mathcal{W}^{\perp}$. Let $(\mathbf{W},{\bf w})$ be a fusion frame
for $\mathcal{W}$ and $({\bf V},{\bf v})$ be a fusion frame for
$\mathcal{V}$. We say that $(\mathbf{V},{\bf v})$ is an oblique dual
fusion frame of $(\mathbf{W},{\bf w})$ on $\mathcal{V}$ if there
exists $Q \in L(K_{\mathcal{W}}, K_{\mathcal{V}})$ such that
\begin{equation}\label{E TvQTw*=I}
T_{{\bf V},{\bf v}}QT^{*}_{\mathbf{W},{\bf
w}}=\pi_{\mathcal{V},\mathcal{W}^{\perp}}.
\end{equation}
\end{defn}
The operator $Q$ is actually important in the definition. If we need
to do an explicit reference to it we say that $({\bf V},{\bf v})$ is
a $Q$-oblique dual fusion frame of $(\mathbf{W},{\bf w})$. Note that
if $({\bf V},{\bf v})$ is a $Q$-oblique dual fusion frame of
$(\mathbf{W},{\bf w})$ on $\mathcal{V}$, then $({\bf W},{\bf w})$ is
a $Q^{*}$-oblique dual fusion frame of $(\mathbf{V},{\bf v})$ on
$\mathcal{W}$. As we will see in Lemma~\ref{L equivalencias}, Bessel
fusion sequences $(\mathbf{W},{\bf w})$ for $\mathcal{W}$ and $({\bf
V},{\bf v})$ for $\mathcal{V}$ that satisfy (\ref{E TvQTw*=I}), are
automatically fusion frames.

As a consequence of Definition~\ref{D oblique fusion frame dual} and
Theorem~\ref{T consistent reconstruction}, we obtain the following
result which establishes that oblique duality yields consistent
reconstruction.

\begin{cor}\label{C reconstruccion consistente sii dual oblicuo}
Let $\mathcal{W}$ and $\mathcal{V}$ be two closed subspaces of
$\mathcal{H}$ such that $\mathcal{H}=\mathcal{V}\oplus
\mathcal{W}^{\perp}$. Let $(\mathbf{W},{\bf w})$ be a fusion frame
for $\mathcal{W}$, $({\bf V},{\bf v})$ be a fusion frame for
$\mathcal{V}$ and $Q \in L(K_{\mathcal{W}}, K_{\mathcal{V}})$. Then
$\widehat{f}:=T_{{\bf V},{\bf v}}QT^{*}_{\mathbf{W},{\bf w}}f$ is a
consistent reconstruction of $f$ for all $f \in \mathcal{H}$ if and
only if $({\bf V},{\bf v})$ is a $Q$-oblique dual fusion frame of
$(\mathbf{W},{\bf w})$ on $\mathcal{V}$.
\end{cor}

It is worth to mention that one reason to introduce first a general
class of oblique dual fusion frames as in Definition~\ref{D oblique
fusion frame dual}, requiring only boundedness of the operator $Q$,
is to ask for the minimal conditions needed to obtain the different
desired properties for oblique dual fusion frames. In particular,
for this general class we have consistent reconstruction as
Corollary~\ref{C reconstruccion consistente sii dual oblicuo} shows
and the following lemma, which generalizes the basic properties that
are valid for dual and oblique dual frames. It is analogous to
\cite[Lemma 3.2]{Heineken-Morillas-Benavente-Zakowicz (2014)} and
gives equivalent conditions for two Bessel fusion sequences to be
oblique dual fusion frames.

\begin{lem}\label{L equivalencias}
Let $\mathcal{W}$ and $\mathcal{V}$ be two closed subspaces of
$\mathcal{H}$ such that $\mathcal{H}=\mathcal{V}\oplus
\mathcal{W}^{\perp} $. Let $(\mathbf{W}, {\bf w})$ be a Bessel
fusion sequence for $\mathcal{W}$, $(\mathbf{V}, {\bf v})$ be a
Bessel fusion sequence for $\mathcal{V}$, and let $Q \in
L(K_{\mathcal{W}}, K_{\mathcal{V}})$. Then the following statements
are equivalent:
\begin{enumerate}
  \item $T_{{\bf V},{\bf v}}Q T^{*}_{{\bf W},{\bf w}}f=f$   for all $f\in \mathcal{V}.$
  \item $T_{{\bf W},{\bf w}}Q^{*}T^{*}_{{\bf V},{\bf v}}f=f$ for all $f\in
  \mathcal{W}.$
  \item $\pi_{\mathcal{V},{\mathcal{W}}^\bot}f=T_{{\bf V},{\bf v}}Q T^{*}_{{\bf W},{\bf w}}f$ for all $f\in \mathcal{H}.$
  \item $\pi_{\mathcal{W},{\mathcal{V}}^\bot}f=T_{{\bf W},{\bf w}}Q^* T^{*}_{{\bf V},{\bf v}}f$ for all $f\in \mathcal{H}.$
  \item $\prode{\pi_{\mathcal{W},{\mathcal{V}}^\bot} f,g}=\prode{Q^*T^{*}_{{\bf V},{\bf v}}f,T_{{\bf W},{\bf w}}^{*}g}$
  for all $f, g \in \mathcal{H}$.
 \item $\prode{\pi_{\mathcal{V},{\mathcal{W}}^\bot} f,g}=\prode{QT^{*}_{{\bf W},{\bf w}}f,T_{{\bf V},{\bf v}}^{*}g}$
  for all $f, g \in \mathcal{H}$.

 \item $T^{*}_{{\bf W},{\bf w}}|_{\mathcal{V}}$ is injective, $T_{{\bf V},{\bf v}}Q$ is surjective and
  $\paren{T^{*}_{{\bf W},{\bf w}}T_{{\bf V},{\bf v}}Q}^{2}=T^{*}_{{\bf W},{\bf w}}T_{{\bf V},{\bf v}}Q$.
 \item $T^{*}_{{\bf V},{\bf v}}|_{\mathcal{W}}$ is injective, $T_{{\bf W},{\bf w}}Q^*$ is surjective and
  $\paren{T^{*}_{{\bf V},{\bf v}}T_{{\bf W},{\bf w}}Q^*}^{2}=T^{*}_{{\bf V},{\bf v}}T_{{\bf W},{\bf w}}Q^*$.
\end{enumerate}
In case any of these equivalent conditions are satisfied,
$(\mathbf{W}, {\bf w})$ is a fusion frame for $\mathcal{W}$,
$(\mathbf{V}, {\bf v})$ is a fusion frame for $\mathcal{V}$,
$(\mathbf{V}, {\bf v})$  is a $Q$-oblique dual fusion frame of
$(\mathbf{W}, {\bf w})$ on $\mathcal{V},$ and $(\mathbf{W}, {\bf
w})$ is an oblique $Q^{*}$-dual fusion frame of $(\mathbf{V}, {\bf
v})$ on $\mathcal{W}.$
\end{lem}

\begin{proof}

$(1) \Leftrightarrow (3)$ and $(2) \Leftrightarrow (3)$ are
immediate.

$(3) \Rightarrow (4): T_{{\bf V},{\bf v}}Q T^{*}_{{\bf W},{\bf
w}}=\pi_{\mathcal{V},{\mathcal{W}}^\bot}.$ Taking adjoint $T_{{\bf
W},{\bf w}}Q^* T^{*}_{{\bf V},{\bf
v}}=\pi_{\mathcal{V},{\mathcal{W}}^\bot}^*.$ But
$\pi_{\mathcal{V},{\mathcal{W}}^\bot}^*=\pi_{\mathcal{W},{\mathcal{V}}^\bot},$
hence $(3)$ follows.

$(4) \Rightarrow (5)$ is clear as well as $(3) \Rightarrow (6).$

$(5) \Rightarrow (4)$: For $f\in \mathcal{H},\,\,\,T_{{\bf W},{\bf
w}}Q^* T^{*}_{{\bf V},{\bf v}}f=T_{{\bf W},{\bf w}}Q^* T^{*}_{{\bf
V},{\bf v}}\pi_{\mathcal{V}}f$ is well defined since $(\mathbf{W},
{\bf w})$ is a Bessel fusion sequence for $\mathcal{W}$ and
$(\mathbf{V},{\bf v})$ is a Bessel fusion sequence for
$\mathcal{V}$. By $(5)$,

\centerline{$\prode{\pi_{\mathcal{W},{\mathcal{V}}^\bot} f-T_{{\bf
W},{\bf w}}Q^*T^{*}_{{\bf V},{\bf v}}f,g}=0\,\,\,\text{ for all } g
\in \mathcal{H},$}

\noindent and so we obtain $(4).$ Analogously $(6) \Rightarrow (3).$

$(1) \Leftrightarrow (7)$: By $(1)$, $T^{*}_{{\bf W},{\bf
w}}|_{\mathcal{V}}$ is injective, $T_{{\bf V},{\bf v}}Q$ is
surjective and

\centerline{$\paren{T^{*}_{{\bf W},{\bf w}}T_{{\bf V},{\bf
v}}Q}^{2}= T^{*}_{{\bf W},{\bf w}}\paren{T_{{\bf V},{\bf
v}}QT^{*}_{{\bf W},{\bf w}}}T_{{\bf V},{\bf v}}Q =T^{*}_{{\bf
W},{\bf w}}T_{{\bf V},{\bf v}}Q.$}

\noindent $(7) \Rightarrow (1)$: If $(T^{*}_{{\bf W},{\bf w}}T_{{\bf
V},{\bf v}}Q)^{2}=T^{*}_{{\bf W},{\bf w}}T_{{\bf V},{\bf v}}Q$ then
$K_{\mathcal{W}}={\rm Ker}(T^{*}_{{\bf W},{\bf w}}T_{{\bf V},{\bf
v}}Q)\oplus{\rm Im}(T^{*}_{{\bf W},{\bf w}}T_{{\bf V},{\bf v}}Q).$
Since $T^{*}_{{\bf W},{\bf w}}|_{\mathcal{V}}$ is injective we have
${\rm Ker}(T^{*}_{{\bf W},{\bf w}}T_{{\bf V},{\bf v}}Q)={\rm
Ker}(T_{{\bf V},{\bf v}}Q)$ and so $K_{\mathcal{W}}={\rm
Ker}(T_{{\bf V},{\bf v}}Q)\oplus{\rm Im}(T^{*}_{{\bf W},{\bf
w}}T_{{\bf V},{\bf v}}Q).$ Therefore, since ${\rm Im}(T_{{\bf
V},{\bf v}}Q)=\mathcal{V},$

\centerline{$\mathcal{V}=\set{T_{{\bf V},{\bf v}}Q\set{f_{i}}_{i \in
I}: \set{f_{i}}_{i \in I} \in {\rm Im}(T^{*}_{{\bf W},{\bf
w}}T_{{\bf V},{\bf v}}Q)}.$}

\noindent Let now $f \in \mathcal{V}$ with $f=T_{{\bf V},{\bf
v}}Q\set{f_{i}}_{i \in I}$ for some $\set{f_{i}}_{i \in I} \in {\rm
Im}(T^{*}_{{\bf W},{\bf w}}T_{{\bf V},{\bf v}}Q)$. Then

\centerline{$T_{{\bf V},{\bf v}}QT^{*}_{{\bf W},{\bf w}}f= T_{{\bf
V},{\bf v}}QT^{*}_{{\bf W},{\bf w}}T_{{\bf V},{\bf
v}}Q\set{f_{i}}_{i \in I}=T_{{\bf V},{\bf v}}Q\set{f_{i}}_{i \in
I}=f.$}

In a similar way it can be proved that $(3) \Leftrightarrow (8)$.

If $(1)$ is satisfied then $T_{{\bf V},{\bf v}}$  is onto and hence
$(\mathbf{V}, {\bf v})$ is a fusion frame for $\mathcal{V}$.
Similarly $(\mathbf{W}, {\bf w})$ is a a fusion frame for
$\mathcal{W}.$
\end{proof}

We will now present two special types of linear transformations $Q$
that make the reconstruction formula that follows from (\ref{E
TvQTw*=I}) simpler. In order to do that we need the selfadjoint
operator $M_{J,\mathbf{W}} : K_{\mathcal{W}} \rightarrow
K_{\mathcal{W}}, M_{J,\mathbf{W}}(f_i)_{i \in
I}=(\chi_{J}(i)f_{i})_{i \in I}.$ We just write $M_{J}$ if it clear
to which $\mathbf{W}$ we refer to. We use the notation
$M_{\{j\},\mathbf{W}}=M_{j,\mathbf{W}}$ and $M_{\{j\}}=M_{j}$.

\begin{defn}Let $Q \in L(K_{\mathcal{W}}, K_{\mathcal{V}})$.
\begin{enumerate}
\item If $QM_{j,\mathbf{W}}K_{\mathcal{W}}
\subseteq M_{j, {\bf V}}K_{\mathcal{V}}$ for each $j \in I,$ $Q$ is
called \emph{block-diagonal}.
\item If $QM_{j,\mathbf{W}}K_{\mathcal{W}}
= M_{j, {\bf V}}K_{\mathcal{V}}$ for each $j \in I,$ $Q$ is called
{\em component preserving}.
\end{enumerate}
\end{defn}

Note that $Q$ is block-diagonal if and only if
$QM_{J,\mathbf{W}}=M_{J,\mathbf{V}}Q$ for each $J \subseteq I$, or
equivalently, $QM_{j,\mathbf{W}}=M_{j, {\bf V}}Q$ for each $j \in
I$. Observe that if $Q$ is block-diagonal, then $Q^{*}$ is
block-diagonal. If in Definition~\ref{D oblique fusion frame dual}
$Q$ is block-diagonal (component preserving) we say that $({\bf
V},{\bf v})$ is a \emph{block-diagonal dual fusion frame}
(\emph{component preserving dual fusion frame}) of $(\mathbf{W},{\bf
w})$.

Another motivation for introducing the notion of oblique duality as
in Definition~\ref{D oblique fusion frame dual} is to obtain
flexibility, therefore asking for restrictions only when needed. The
general framework provided by Definition~\ref{D oblique fusion frame
dual} allows to adjust to the problem at hand. This is another
reason to start with the most general class and then naturally arise
the particular classes with which we work here: block-diagonal and
component preserving oblique dual fusion frames. As we will see in
Lemma~\ref{L Vi=ApiWj entonces V,v dual fusion frame}, $Q$ is
component preserving for oblique dual fusion frames obtained from
the oblique left inverses of $T^{*}_{\mathbf{W},{\bf w}}$. Also, $Q$
is block-diagonal for oblique dual fusion frame systems (see
Definition~\ref{D oblique dual fusion frame system} and
Remark~\ref{R Q Mi sum Wj subset Mi sum Vj}).

If $Q$ is block-diagonal, from (\ref{E TvQTw*=I}) we obtain the
following reconstruction formula:
\begin{equation}\label{E fcs}
\pi_{\mathcal{V},\mathcal{W}^{\perp}}f=\sum_{j \in
I}v_{j}w_{j}Q_{j}f~,~~\forall f \in \mathcal{H},\end{equation} where
$Q_{j} : \mathcal{H} \rightarrow V_{j}$ is given by
$Q_{j}f:=(QM_{j}(\pi_{W_{i}}f)_{i \in I})_{j}$. For each $j \in I$,
$Q_{j}$ is a bounded linear operator. Observe that $W_{j}^{\bot}
\subseteq \text{ker}(Q_{j})$ and we can recover the block-diagonal
(or component preserving) mapping $Q$ as $Q=\oplus_{j \in I}Q_{j}$.

Note that  $({\bf V},{\bf v})$ is a $Q$-oblique dual fusion frame of
$(\textbf{W},{\bf w})$ if and only if $({\bf V},{\bf c v})$ is a
$\frac{1}{\mathbf{c}}Q$-oblique dual fusion frame of
$(\textbf{W},{\bf w})$, where ${\bf c v}=\{c_{i}w_{i}\}_{i \in I}$
and $\frac{1}{\mathbf{c}}Q=\oplus_{j \in I}(\frac{1}{c_{j}}Q_{j})$,
with $0 < \inf_{i \in I}c_{i} \leq \sup_{i \in I}c_{i} < \infty$.
Both oblique dual fusion frames lead to the same reconstruction
formula. This freedom for the weights is desirable because we can
select those ${\bf v}$ such that the pair $({\bf V},{\bf v})$ is the
most suitable to treat simultaneously another problem not related
with the reconstruction formula.

We observe that in each term of (\ref{E fcs}) we can think the
product $v_{j}w_{j}||Q_{j}||$ as the weight that accompanies the
pair of subspaces $W_{j}$ and $V_{j},$ determining their importance
in the reconstruction.

\subsection{Oblique dual fusion frame systems}

We will define and study in this section the concept of oblique dual
fusion frame systems. In order to do that, we will need the
following operator, which we introduced in \cite{Heineken-Morillas
(2014)}, and which establishes the connection between the synthesis
operator of a fusion frame system and the synthesis operator of its
associated frame.

Let $(\mathbf{W}, {\bf w})$ be a Bessel fusion sequence for
$\mathcal{W}$ and $\mathcal{F}_{i}$ be a frame for $W_i$ with frame
bounds $\alpha_{i}, \beta_{i}$ for each $i \in I$ such that ${\rm
sup}_{i \in I}\beta_{i}= \beta < \infty$. Let

\centerline{$C_{\mathcal{F}}: \oplus_{i \in I}
\ell^2(L_{i})\rightarrow K_{\mathcal{W}},\,\,
C_{\mathcal{F}}((x_{i,l})_{l \in L_i})_{i \in
I}=(T_{\mathcal{F}_{i}}(x_{i,l})_{l \in L_i})_{i \in I}.$}

Note that $C_{\mathcal{F}}$ is a surjective bounded operator with
$||C_{\mathcal{F}}|| \leq \beta$. Its adjoint is
$C_{\mathcal{F}}^{*}: K_{\mathcal{W}} \rightarrow \oplus_{i \in I}
\ell^2(L_{i})$, given by $C_{\mathcal{F}}^{*}(g_{i})_{i \in
I}=(T_{\mathcal{F}_{i}}^{*}g_{i})_{i \in I}$ and satisfies
$||C_{\mathcal{F}}^{*}(g_{i})_{i \in I}|| \leq \beta||(g_{i})_{i \in
I}||$. If $0 < \alpha ={\rm inf}_{i \in I}\alpha_{i}$, we also have
$\alpha ||(g_{i})_{i \in I}|| \leq ||C_{\mathcal{F}}^{*}(g_{i})_{i
\in I}||$. The bounded left inverses of $C_{\mathcal{F}}^{*}$ are
all $C_{\widetilde{\mathcal{F}}} \in L(\oplus_{i \in I}
\ell^2(L_{i}), K_{\mathcal{W}})$ such that
$\widetilde{\mathcal{F}}_{i}$ is a dual frame of $\mathcal{F}_{i}$
with upper frame bound $\widetilde{\beta}_{i}$ for each $i \in I$
such that ${\rm sup}_{i \in I}\widetilde{\beta}_{i} < \infty$.
Observe that

\centerline{$T_{{\bf w}\mathcal{F}}=T_{\mathbf{W},{\bf
w}}C_{\mathcal{F}}\,\,\text { and }\,\,T_{\mathbf{W},{\bf
w}}=T_{{\bf w}\mathcal{F}}C_{\widetilde{\mathcal{F}}}^{*}.$}

We define oblique dual fusion frame systems as follows:
\begin{defn}\label{D oblique dual fusion frame system}
Let $\mathcal{W}$ and $\mathcal{V}$ be two closed subspaces of
$\mathcal{H}$ such that
$\mathcal{H}=\mathcal{V}\oplus\mathcal{W}^{\perp}$. Let
$(\mathbf{W},{\bf w}, \mathcal{F})$ be a fusion frame system for
$\mathcal{W}$ with upper local frame  bound $\beta_{i}$ for each $i
\in I$ such that ${\rm sup}_{i \in I}\beta_{i} < \infty$, $({\bf
V},{\bf v}, \mathcal{G})$ be a fusion frame system for $\mathcal{V}$
with local upper frame bound $\widetilde{\beta}_{i}$ for each $i \in
I$ such that ${\rm sup}_{i \in I}\widetilde{\beta}_{i} < \infty$ and
$|\mathcal{F}_{i}|=|\mathcal{G}_{i}|$ for each $i \in I$. Then
$({\bf V},{\bf v}, \mathcal{G})$ is an oblique dual fusion frame
system of $(\mathbf{W},{\bf w}, \mathcal{F})$ on $\mathcal{V}$ if
$({\bf V},{\bf v})$ is a
$C_{\mathcal{G}}C_{\mathcal{F}}^{*}$-oblique dual fusion frame of
$(\mathbf{W},{\bf w})$ on $\mathcal{V}$.
\end{defn}
\begin{rem}\label{R Q Mi sum Wj subset Mi sum Vj}
It is easy to see that the operator
$C_{\mathcal{G}}C_{\mathcal{F}}^{*}: K_{\mathcal{W}} \rightarrow
K_{\mathcal{V}},\, C_{\mathcal{G}}C_{\mathcal{F}}^{*}(f_{i})_{i \in
I}=(T_{\mathcal{G}_{i}}T_{\mathcal{F}_{i}}^{*}f_{i})_{i \in I}$ is
block-diagonal.
\end{rem}
If $C_{\mathcal{G}}C_{\mathcal{F}}^{*}$ in Definition~\ref{D oblique
dual fusion frame system}  is component preserving, we call $({\bf
V},{\bf v}, \mathcal{G})$ a \emph{component preserving oblique dual
fusion frame system} of $(\mathbf{W},{\bf w}, \mathcal{F}).$
\begin{rem}
If $\mathcal{W}=\mathcal{V}$ we have in Definition~\ref{D oblique
fusion frame dual} and in Definition~\ref{D oblique dual fusion
frame system} the concepts of dual fusion frame and dual fusion
frame system, respectively, considered in
\cite{Heineken-Morillas-Benavente-Zakowicz (2014), Heineken-Morillas
(2014)}. In this case, we simply say that $({\bf V},{\bf v})$ is a
$Q$- dual fusion frame of $(\mathbf{W},{\bf w})$ or that $({\bf
V},{\bf v}, \mathcal{G})$ is a dual fusion frame system of
$(\mathbf{W},{\bf w}, \mathcal{F})$.
\end{rem}

\section{Relation between block-diagonal oblique dual frames,
oblique dual fusion frame systems and oblique dual frames}

We defined oblique dual fusion frame systems in terms of
(block-diagonal) oblique dual fusion frames (see Definition~\ref{D
oblique dual fusion frame system} and Remark~\ref{R Q Mi sum Wj
subset Mi sum Vj}). Conversely, we can always associate to a
block-diagonal oblique dual fusion frame pair an oblique dual fusion
frame system pair. In order to see this we need the following two
auxiliary results.

\begin{lem}\label{L A=CGCF*}
If $A \in L(\mathcal{H}, \mathcal{K})$, then there exists a frame
$\mathcal{F}$ for $\mathcal{H}$ and a frame $\mathcal{G}$ for
$\mathcal{K}$ such that $|\mathcal{F}|=|\mathcal{G}|$ and
$A=T_{\mathcal{G}}T_{\mathcal{F}}^{*}$. We can choose $\mathcal{F}$
and $\mathcal{G}$ in such a way that their frame bounds are $1, 2$
and $1,1+||A||^{2}$, respectively.
\end{lem}

\begin{proof}
Let $\mathcal{F}$ be any frame for $\mathcal{H}$ and
$\widetilde{\mathcal{F}}$ be any dual frame of $\mathcal{F}$. Then
$A\widetilde{\mathcal{F}}$ is a frame for ${\rm Im}(A)$,
$|A\widetilde{\mathcal{F}}|=|\mathcal{F}|$ and
$A=T_{A\widetilde{\mathcal{F}}}T_{\mathcal{F}}^{*}$.

If ${\rm Im}(A)\neq \mathcal{K}$ let $\mathcal{G}=\{g_{i}\}_{i \in
J}$ be any frame for ${\rm Im}(A)^{\perp}$. If
$\widetilde{\mathcal{G}}$ is the family indexed by $J$ with all its
elements equal to the zero of $\mathcal{H}$, then
$T_{\mathcal{G}}T_{\widetilde{\mathcal{G}}}^{*}=0$. We can also
construct $\widetilde{\mathcal{G}}$ with not all of its elements
equal to zero that has this property. For this, we consider a frame
$\mathcal{G}=\{g_{j}\}_{j \in J}$ for ${\rm Im}(A)^{\perp}$ that is
not a basis. Let $\{c_{m}\}_{m \in \mathbb{M}}$ be an orthonormal
basis for ${\rm Ker}(T_{\mathcal{G}}) \subset \ell^{2}(J)$ where
$\mathbb{M}=\mathbb{N}$ or $\mathbb{M}=\{1, \ldots, M\}$. Let
$\{e_{l}\}_{l \in \mathbb{L}}$ be an orthonormal basis for
$\mathcal{H}$ where $\mathbb{L}=\mathbb{N}$ or $\mathbb{L}=\{1,
\ldots, L\}$. Let $\mathbb{I}$ any finite subset of $\mathbb{M} \cap
\mathbb{L}$ and $\widetilde{g}_{j}=\sum_{l \in
\mathbb{I}}\overline{c_{l}(j)}e_{l}$ for each $j \in J$. By the
linear independence of $\{e_{l}\}_{l \in \mathbb{L}}$ and
$\{c_{m}\}_{m \in \mathbb{M}}$, the vectors $\widetilde{g}_{j}$ can
not be all equal to $0$. We have $\sum_{j \in J}|\langle f,
\widetilde{g}_{j}\rangle|^{2} \leq \sum_{l \in \mathbb{I}}|\langle
f, e_{l}\rangle|^{2}\sum_{j \in J}|c_{l}(j)|^{2} \leq ||f||^{2}$.
Therefore, $\{\widetilde{g}_{j}\}_{j \in J}$ is a Bessel sequence
with Bessel bound $1$. Note that if $d \in \ell^{2}(J)$ and $f \in
\mathcal{H}$, then $\langle T_{\widetilde{\mathcal{G}}}d, f \rangle
= \langle \sum_{j \in J}d(j)\sum_{l \in
\mathbb{I}}\overline{c_{l}(j)}e_{l}, f\rangle = \sum_{j \in
J}d(j)\sum_{l \in \mathbb{I}}\overline{c_{l}(j)}\langle e_{l},
f\rangle = \langle d , \sum_{l \in \mathbb{I}}\langle f,
e_{l}\rangle c_{l}\rangle$, and then
$T_{\mathcal{G}}T_{\widetilde{\mathcal{G}}}^{*}f=\sum_{l \in
\mathbb{I}}\langle f, e_{l}\rangle T_{\mathcal{G}}c_{l}=0$ since
$c_{l} \in {\rm Ker}(T_{\mathcal{G}})$ for each $l \in \mathbb{I}$.

Finally, $\{\mathcal{F},\widetilde{\mathcal{G}}\}$ is a frame for
$\mathcal{H}$, $\{A\widetilde{\mathcal{F}},\mathcal{G}\}$ is a frame
for $\mathcal{K}$,
$|\{\mathcal{F},\widetilde{\mathcal{G}}\}|=|\{A\widetilde{\mathcal{F}},\mathcal{G}\}|$
and
$A=T_{\{A\widetilde{\mathcal{F}},\mathcal{G}\}}T_{\{\mathcal{F},\widetilde{\mathcal{G}}\}}^{*}$.

If we choose $\mathcal{F}$, $\widetilde{\mathcal{F}}$ and
$\mathcal{G}$ to be Parseval, the frame bounds of
$\{\mathcal{F},\widetilde{\mathcal{G}}\}$ and
$\{A\widetilde{\mathcal{F}},\mathcal{G}\}$ are $1, 2$ and
$1,1+||A||^{2}$, respectively.
\end{proof}

\begin{cor}\label{C Q=CGCF*}
If $Q \in L(K_{\mathcal{W}}, K_{\mathcal{V}})$ is block-diagonal
then there exists a frame $\mathcal{F}_{i}$ for $W_i$ with frame
bounds $\alpha_{i}, \beta_{i}$ for each $i \in I$, satisfying $ 0 <
{\rm inf}_{i \in I}\alpha_{i} \leq {\rm sup}_{i \in I}\beta_{i} <
\infty$, and a frame $\mathcal{G}_{i}$ for $V_i$ with
$|\mathcal{F}_{i}|=|\mathcal{G}_{i}|$ and frame bounds
$\widetilde{\alpha}_{i}, \widetilde{\beta}_{i}$ for each $i \in I$,
satisfying $ 0 < {\rm inf}_{i \in I}\widetilde{\alpha}_{i} \leq {\rm
sup}_{i \in I}\widetilde{\beta}_{i} < \infty$ such that
$Q=C_{\mathcal{G}}C_{\mathcal{F}}^{*}$.
\end{cor}

\begin{proof}
By Lemma~\ref{L A=CGCF*}, for each $i \in I$ there exists frames
$\mathcal{F}_{i}$ for $W_i$ with frame bounds $\alpha_{i}=1$,
$\beta_{i}=2$ and $\mathcal{G}_{i}$ for $V_i$ with frame bounds
$\widetilde{\alpha}_{i}=1$, $\widetilde{\beta}_{i}=1+||Q_{i}||^{2}$
such that $Q_{i}=T_{\mathcal{G}_{i}}T_{\mathcal{F}_{i}}^{*}$. Thus,
$Q=C_{\mathcal{G}}C_{\mathcal{F}}^{*}$. Moreover, $||Q_{i}|| \leq
||Q||$ for each $i \in I$.
\end{proof}

The next theorem asserts that a block-diagonal oblique dual fusion
frame pair can always be viewed as an oblique dual fusion frame
system pair.

\begin{thm}\label{T relacion entre dual y dual system}
Let $\mathcal{W}$ and $\mathcal{V}$ be two closed subspaces of
$\mathcal{H}$. Let $(\mathbf{W},{\bf w})$ be a fusion frame for
$\mathcal{W}$ and let $({\bf V},{\bf v})$ be a block diagonal
$Q$-oblique dual fusion frame of $(\mathbf{W},{\bf w})$ on
$\mathcal{V}$. Then there exists a frame $\mathcal{F}_{i}$ for $W_i$
with frame bounds $\alpha_{i}, \beta_{i}$ for each $i \in I$ such
that $0 < {\rm inf}_{i \in I}\alpha_{i} \leq {\rm sup}_{i \in
I}\beta_{i} < \infty$ and a frame $\mathcal{G}_{i}$ for $V_i$ with
frame bounds $\widetilde{\alpha}_{i}, \widetilde{\beta}_{i}$ for
each $i \in I$ such that $ 0 < {\rm inf}_{i \in
I}\widetilde{\alpha}_{i} \leq {\rm sup}_{i \in
I}\widetilde{\beta}_{i} < \infty$, such that $({\bf V},{\bf
v},\mathcal{G})$ is a dual fusion frame system of $({\bf W},{\bf
w},\mathcal{F})$ and $Q=C_{\mathcal{G}}C_{\mathcal{F}}^{*}$.
\end{thm}

\begin{proof}
It is a consequence of Definition~\ref{D oblique fusion frame dual},
Corollary~\ref{C Q=CGCF*} and Definition~\ref{D oblique dual fusion
frame system}.
\end{proof}

The following theorem establishes the connection between the notions
of oblique dual fusion frame system and oblique dual frame.

\begin{thm}\label{T dual fusion frame systems} Let $\mathcal{W}$
and $\mathcal{V}$ be two closed subspaces of $\mathcal{H}$ such that
$\mathcal{H}=\mathcal{V}\oplus\mathcal{W}^{\perp}$. Let
$(\mathbf{W},{\bf w}, \mathcal{F})$ be a Bessel fusion system for
$\mathcal{W}$ such that $\mathcal{F}_{i}$ has upper frame bound
$\beta_{i}$ for each $i \in I$ with ${\rm sup}_{i \in I}\beta_{i} <
\infty$, and let $({\bf V},{\bf v}, \mathcal{G})$ be a Bessel fusion
system for $\mathcal{V}$  such that $\mathcal{G}_{i}$ has upper
frame bound $\widetilde{\beta}_{i}$ for each $i \in I$ with ${\rm
sup}_{i \in I}\widetilde{\beta}_{i}< \infty$. If
$|\mathcal{F}_{i}|=|\mathcal{G}_{i}|$ for each $i \in I$ then the
following conditions are equivalent:
  \begin{enumerate}
    \item ${\bf w}\mathcal{F}$ and ${\bf v}\mathcal{G}$ are oblique dual frames for $\mathcal{H}.$
    \item  $(\mathbf{V}, {\bf v},\mathcal{G})$ is an oblique dual fusion frame system of $(\mathbf{W}, {\bf w},\mathcal{F})$ on $\mathcal{V}$.
  \end{enumerate}
\end{thm}

\begin{proof}
It follows from Theorem~\ref{T wF marco sii WwF fusion frame
system}, the equality $T_{{\bf v}\mathcal{G}}T_{{\bf
w}\mathcal{F}}^{*}=T_{{\bf V},{\bf
v}}C_{\mathcal{G}}C_{\mathcal{F}}^{*}T_{\mathbf{W},{\bf w}}^{*}$,
\cite[Lemma 6.3.2]{Christensen (2016)} and Lemma~\ref{L
equivalencias}.
\end{proof}

%------------------------------------------------------------

\section{Duals and oblique duals}

Oblique duality is a generalization of duality. The results in this
section provide methods to obtain dual fusion frames from oblique
dual fusion frmaes and viceversa.

\begin{prop}\label{P oblique dual system entonces dual system}
Let $\mathcal{W}$ and $\mathcal{V}$ be two closed subspaces of
$\mathcal{H}$ such that
$\mathcal{H}=\mathcal{V}\oplus\mathcal{W}^{\perp}$. Let
$(\mathbf{W},{\bf w}, \mathcal{F})$ be a fusion frame system for
$\mathcal{W}$ with local upper frame bounds $\beta_{i}$ for each $i
\in I$ such that ${\rm sup}_{i \in I}\beta_{i} < \infty$, $({\bf
V},{\bf v}, \mathcal{G})$ be a fusion frame system for $\mathcal{V}$
with local upper frame bounds $\widetilde{\beta}_{i}$ for each $i
\in I$ such that ${\rm sup}_{i \in I}\widetilde{\beta}_{i} < \infty$
and $|\mathcal{F}_{i}|=|\mathcal{G}_{i}|$ for each $i \in I$. If
$({\bf V},{\bf v}, \mathcal{G})$  is an oblique dual fusion frame
system of $(\mathbf{W},{\bf w}, \mathcal{F})$ on $\mathcal{V}$, then
$(\pi_{\mathcal{W}}(\mathbf{V}),{\bf
v},\pi_{\mathcal{W}}(\mathcal{G}))$
 is a dual fusion frame system of
$(\mathbf{W},{\bf w}, \mathcal{F})$  for $\mathcal{W}$ and
$(\pi_{\mathcal{V}}(\mathbf{W}),{\bf
w},\pi_{\mathcal{V}}(\mathcal{F}))$ is a dual fusion frame system of
 $({\bf V},{\bf v}, \mathcal{G})$ for $\mathcal{V}.$
\end{prop}

\begin{proof}
Since $\mathcal{V} \cap \mathcal{W}^{\perp} = \{0\}$, if $f \in
\mathcal{V}$ is such that $\pi_{\mathcal{W}}f=0$, then $f=0$. So,
$\pi_{\mathcal{W}}(\mathcal{V})\neq\{0\}$. By \cite[Proposition
5.3.1]{Christensen (2016)}, $\pi_{\mathcal{W}}(\mathcal{G}_i)$ is a
frame for $\pi_{\mathcal{W}}(V_i)$ with upper frame bound
$\widetilde{\beta}_{i}$ for each $i \in I$ and
$\pi_{\mathcal{W}}(\mathcal{G})$ is a frame for
$\pi_{\mathcal{W}}(\mathcal{V})$. By Theorem~\ref{T wF marco sii WwF
fusion frame system}, $(\pi_{\mathcal{W}}(\mathbf{V}),{\bf
v},\pi_{\mathcal{W}}(\mathcal{G}))$ is a Bessel fusion sequence for
$\pi_{\mathcal{W}}(\mathcal{V})$. We have $T_{{\bf V},{\bf
v}}C_{\mathcal{G}}C_{\mathcal{F}}^{*}T^{*}_{{\bf W},{\bf
w}}=\pi_{\mathcal{V},{\mathcal{W}}^\bot}$. Also
$T_{\pi_{\mathcal{W}}(\mathbf{V}),{\bf
v}}C_{\pi_{\mathcal{W}}(\mathcal{G})}=\pi_{\mathcal{W}}T_{{\bf
v}\mathcal{G}}.$ Hence,

\centerline{$T_{\pi_{\mathcal{W}}(\mathbf{V}),{\bf
v}}C_{\pi_{\mathcal{W}}(\mathcal{G})}C_{\mathcal{F}}^{*}T^{*}_{{\bf
W},{\bf w}}=\pi_{\mathcal{W}}T_{{\bf
v}\mathcal{G}}C_{\mathcal{F}}^{*}T^{*}_{{\bf W},{\bf
w}}=\pi_{\mathcal{W}}T_{{\bf V},{\bf
v}}C_{\mathcal{G}}C_{\mathcal{F}}^{*}T^{*}_{{\bf W},{\bf
w}}=\pi_{\mathcal{W}}\pi_{\mathcal{V},{\mathcal{W}}^\bot}=\pi_{\mathcal{W}}.$}

By \cite[Lemma 3.2]{Heineken-Morillas-Benavente-Zakowicz (2014)},
$(\pi_{\mathcal{W}}(\mathbf{V}),{\bf
v},\pi_{\mathcal{W}}(\mathcal{G}))$
 is a dual fusion frame system of
$(\mathbf{W},{\bf w}, \mathcal{F})$  for $\mathcal{W}$.

The other assertion is proved in a similar way.
\end{proof}
By Theorem~\ref{T relacion entre dual y dual system} and
Definition~\ref{D oblique dual fusion frame system} we obtain the
following Corollary:
\begin{cor}\label{C oblique dual entonces dual}
Let $\mathcal{W}$ and $\mathcal{V}$ be two closed subspaces of
$\mathcal{H}$ such that
$\mathcal{H}=\mathcal{V}\oplus\mathcal{W}^{\perp}$. If $({\bf
V},{\bf v})$  is a block-diagonal $\oplus_{i \in I}Q_{i}$-oblique
dual fusion frame of $(\mathbf{W},{\bf w})$ on $\mathcal{V}$, then
$(\pi_{\mathcal{W}}(\mathbf{V}),{\bf v})$
 is a  $\oplus_{i \in
I}\pi_{\mathcal{W}}Q_{i}$-block-diagonal dual fusion frame of
$(\mathbf{W},{\bf w})$ for $\mathcal{W}$ and
$(\pi_{\mathcal{V}}(\mathbf{W}),{\bf w})$ is a $\oplus_{i \in
I}\pi_{\mathcal{V}}Q_{i}$-block-diagonal dual fusion frame of
 $({\bf V},{\bf v})$ for $\mathcal{V}.$
\end{cor}

The following two results can be proved in a similar way as
Proposition~\ref{P oblique dual system entonces dual system} and
Corollary~\ref{C oblique dual entonces dual}, respectively.

\begin{prop}
Let $\mathcal{W}$ and $\mathcal{V}$ be two closed subspaces of
$\mathcal{H}$ such that
$\mathcal{H}=\mathcal{V}\oplus\mathcal{W}^{\perp}$. Let
$(\mathbf{W},{\bf w}, \mathcal{F})$ be a fusion frame system for
$\mathcal{W}$ with local upper frame bound $\beta_{i}$ for each $i
\in I$ such that ${\rm sup}_{i \in I}\beta_{i} < \infty$,
$(\widetilde{{\bf W}},\widetilde{{\bf w}}, \widetilde{\mathcal{F}})$
be a fusion frame system for $\mathcal{W}$ with local upper frame
bound $\widetilde{\beta}_{i}$ for each $i \in I$ such that ${\rm
sup}_{i \in I}\widetilde{\beta}_{i} < \infty$ and
$|\mathcal{F}_{i}|=|\widetilde{\mathcal{F}}_{i}|$ for each $i \in
I$. If $(\widetilde{{\bf W}},\widetilde{{\bf w}},
\widetilde{\mathcal{F}})$ is a dual fusion frame system of
$(\mathbf{W},{\bf w}, \mathcal{F})$, then
$(\pi_{\mathcal{V},{\mathcal{W}}^\bot}\widetilde{{\bf
W}},\widetilde{{\bf w}},
\pi_{\mathcal{V},{\mathcal{W}}^\bot}\widetilde{\mathcal{F}})$ is an
oblique dual fusion frame system of $(\mathbf{W},{\bf w},
\mathcal{F})$ on $\mathcal{V}$.
\end{prop}

\begin{cor}\label{C dual entonces oblique dual}
Let $\mathcal{W}$ and $\mathcal{V}$ be two closed subspaces of
$\mathcal{H}$ such that
$\mathcal{H}=\mathcal{V}\oplus\mathcal{W}^{\perp}$. If $({\bf
\widetilde{W}},{\bf \widetilde{w}})$ is a block-diagonal $\oplus_{i
\in I}Q_{i}$-dual fusion frame of $(\mathbf{W},{\bf w})$ for
$\mathcal{W}$, then
$(\pi_{\mathcal{V},{\mathcal{W}}^\bot}(\mathbf{\widetilde{W}}),{\bf
\widetilde{w}})$
 is a block-diagonal $\oplus_{i \in
I}\pi_{\mathcal{V},{\mathcal{W}}^\bot}Q_{i}$-oblique dual fusion
frame of $(\mathbf{W},{\bf w})$ on $\mathcal{V}$.
\end{cor}

%%%%%%%%%%

\section{Oblique dual families}

In \cite{Heineken-Morillas-Benavente-Zakowicz (2014),
Heineken-Morillas (2014)} it is shown that component preserving dual
fusion frames are related to the left inverses of the analysis
operator. In this section we will show that analogous results are
valid for component preserving oblique dual fusion frames.

The following Lemma can be deduced from Corollary~\ref{C oblique
dual entonces dual}, \cite[Lemma
3.4]{Heineken-Morillas-Benavente-Zakowicz (2014)}, Lemma~\ref{L
composicion proyeccion oblicua con ortogonal} and Proposition~\ref{P
relacion inversa inversa oblicua}. Nevertheless, we include a short
direct proof of it.

\begin{lem}\label{L V,v dual fusion frame entonces Vi=ApiWj}
Let $\mathcal{W}$ and $\mathcal{V}$ be two closed subspaces of
$\mathcal{H}$ such that
$\mathcal{H}=\mathcal{V}\oplus\mathcal{W}^{\perp}$. Let
$(\mathbf{W},{\bf w})$ be a fusion frame for $\mathcal{W}$. If
$({\bf V},{\bf v})$ is a $Q$-component preserving oblique dual
fusion frame of $(\mathbf{W},{\bf w})$ on $\mathcal{V}$ then
$V_{i}=AM_{i}K_{\mathcal{W}}$ for each $i\in I$ for
$A=T_{\mathbf{V},{\bf v}}Q \in \mathfrak{L}_{T^{*}_{\mathbf{W},{\bf
w}}}^{\mathcal{V},\mathcal{W}^{\perp}}$.
\end{lem}
\begin{proof}
Let $Q\in L(\mathcal{K}_{\mathcal{W}},\mathcal{K}_{\mathcal{V}})$ be
component preserving such that $T_{\mathbf{V},{\bf
v}}QT^{*}_{\mathbf{W},{\bf
w}}=\pi_{\mathcal{V},\mathcal{W}^{\perp}}$ and let
$A=T_{\mathbf{V},{\bf v}}Q$. Clearly,
$A\in\mathfrak{L}_{T^{*}_{\mathbf{W},{\bf
w}}}^{\mathcal{V},\mathcal{W}^{\perp}}$. Using that $Q$ is component
preserving, $AM_{i}K_{\mathcal{W}}=T_{\mathbf{V},{\bf
v}}QM_{i}K_{\mathcal{W}}=V_{i}$ for each $i \in I$.\end{proof}

A reciprocal of the previous result is:

\begin{lem}\label{L Vi=ApiWj entonces V,v dual fusion frame}
Let $\mathcal{W}$ and $\mathcal{V}$ be two closed subspaces of
$\mathcal{H}$ such that
$\mathcal{H}=\mathcal{V}\oplus\mathcal{W}^{\perp}$. Let
$(\mathbf{W},\mathbf{w})$ be a fusion frame for $\mathcal{W}$, $A
\in \mathfrak{L}_{T^{*}_{\mathbf{W},{\bf
w}}}^{\mathcal{V},\mathcal{W}^{\perp}}$ and
$V_{i}=AM_{i}\mathcal{K}_{\mathcal{W}}$ for each $i \in I$. If
$(\mathbf{V},\mathbf{v})$ is a Bessel fusion sequence for
$\mathcal{V}$ and

\centerline{$Q_{A,{\bf v}}: K_{\mathcal{W}} \rightarrow
K_{\mathcal{V}} \text{ , }Q_{A,{\bf v}}(f_j)_{j \in
I}=\paren{\frac{1}{v_{i}}AM_{i}(f_j)_{j \in I}}_{i \in I},$}

\noindent is a well defined bounded operator, then
$(\mathbf{V},\mathbf{v})$ is a $Q_{A,\mathbf{v}}$-component
preserving oblique dual fusion frame of $(\mathbf{W},\mathbf{w})$ on
$\mathcal{V}$.
\end{lem}

\begin{proof}
From the hypotheses, $\overline{\text{span}}\bigcup_{i \in
I}V_{i}={\rm Im}(A)=\mathcal{V}$, $Q_{A,{\bf v}}$ is component
preserving and $A=T_{\mathbf{V},{\bf v}}Q_{A,{\bf v}}$. Since
$A\in\mathfrak{L}_{T^{*}_{\mathbf{W},{\bf
w}}}^{\mathcal{V},\mathcal{W}^{\perp}}$, $T_{\mathbf{V},{\bf
v}}Q_{A,{\bf v}}T^{*}_{\mathbf{W},{\bf
w}}=\pi_{\mathcal{V},\mathcal{W}^{\perp}}$. So $(\mathbf{V},{\bf
v})$ is a $Q_{A,{\bf v}}$-component preserving oblique dual fusion
frame of $(\mathbf{W},{\bf w})$ on $\mathcal{V}$.
\end{proof}

See \cite[Remark 3.6]{Heineken-Morillas-Benavente-Zakowicz (2014)}
for sufficient conditions for $(\mathcal{V},v)$ being a Bessel
fusion sequence for $\mathcal{V}$ and for $Q_{A,v}$ being a well
defined bounded operator in Lemma~\ref{L Vi=ApiWj entonces V,v dual
fusion frame}. For the case in which $\mathcal{W}$ and $\mathcal{V}$
are finite-dimensional, Lemma~\ref{L V,v dual fusion frame entonces
Vi=ApiWj} and Lemma~\ref{L Vi=ApiWj entonces V,v dual fusion frame}
lead to the following characterization:

\begin{thm}\label{T V,v dual fusion
frame sii Vi=ApiWj} Let $\mathcal{W}$ and $\mathcal{V}$ be
finite-dimensional subspaces of the Hilbert space $\mathcal{H}$ such
that $\mathcal{H}=\mathcal{V}\oplus \mathcal{W}^{\perp}$. Let
$\mathcal{W}$ and $\mathcal{V}$ be two subspaces of $\mathcal{H}$
such that $\mathcal{H}=\mathcal{V}\oplus\mathcal{W}^{\perp}$. Let
$(\mathbf{W},{\bf w})$ be a fusion frame for $\mathcal{W}$. Then
$({\bf V},{\bf v})$ is a $Q$-component preserving oblique dual
fusion frame of $(\mathbf{W},{\bf w})$ on $\mathcal{V}$ if and only
if $V_{i}=AM_{i}K_{\mathcal{W}}$ for each $i\in I$ and $Q=Q_{A,{\bf
v}},$ for some $A \in \mathfrak{L}_{T^{*}_{\mathbf{W},{\bf
w}}}^{\mathcal{V},\mathcal{W}^{\perp}}$ with ${\rm
Im}(A)=\mathcal{V}$. Moreover, any element of
$\mathfrak{L}_{T^{*}_{\mathbf{W},{\bf
w}}}^{\mathcal{V},\mathcal{W}^{\perp}}$ with ${\rm
Im}(A)=\mathcal{V}$ is of the form $T_{{\bf V},{\bf v}}Q$ where
$({\bf V},{\bf v})$ is some $Q$-component preserving oblique dual
fusion frame of $(\mathbf{W},{\bf w})$ on $\mathcal{V}$.
\end{thm}

The above results show that component preserving oblique dual fusion
frames can be obtained in a similar manner as in the vectorial case
(see \cite[Lemma 6.3.5]{Christensen (2016)}).

\begin{rem}
By Proposition~\ref{P relacion inversa inversa oblicua},
Theorem~\ref{T V,v dual fusion frame sii Vi=ApiWj} and \cite[Theorem
3.5]{Heineken-Morillas (2014)}, if $\mathcal{W}$ and $\mathcal{V}$
are finite-dimensional, there is a bijection between component
preserving dual and component preserving oblique dual fusion frames.
\end{rem}

\begin{rem}
As a consequence of Theorem~\ref{T V,v dual fusion frame sii
Vi=ApiWj}, if $\mathcal{W}$ and $\mathcal{V}$ are finite-dimensional
we can always associate to any $Q$-oblique dual fusion frame $({\bf
V},{\bf v})$ of $(\mathbf{W},{\bf w})$ the $Q_{A,{\bf
\tilde{v}}}$-component preserving oblique dual fusion frame
$\{(AM_{i}K_{\mathcal{W}},\tilde{v}_{i})\}_{i \in I}$ with
$A=T_{{\bf V},{\bf v}}Q$ and $\{\widetilde{v}_{i}\}_{i \in I}$
arbitrary weights. Furthermore, if $Q$ is block-diagonal, then
$Q_{T_{{\bf V},{\bf v}}Q,{\bf v}}(f_i)_{i \in I}=Q(f_i)_{i \in I}$
for each $(f_i)_{i \in I} \in K_{\mathcal{W}}$.
\end{rem}

\begin{rem}\label{R conjunto inversas oblicuas de TWw*}
It is easy to see that
\begin{align*}
\mathfrak{L}_{T^{*}_{\mathbf{W},{\bf
w}}}^{\mathcal{V},\mathcal{W}^{\perp}}&=\{\pi_{\mathcal{V},\mathcal{W}^{\perp}}S_{\mathbf{W},{\bf
w}}^{\dagger}T_{\mathbf{W},{\bf
w}}+B(I_{K_{\mathcal{W}}}-T^{*}_{\mathbf{W},{\bf
w}}S_{\mathbf{W},{\bf w}}^{\dagger}T_{\mathbf{W},{\bf w}}):B\in
L(K_{\mathcal{W}},\mathcal{V})\}\\&=\{\pi_{\mathcal{V},\mathcal{W}^{\perp}}S_{\mathbf{W},{\bf
w}}^{\dagger}T_{\mathbf{W},{\bf w}}+B:B\in
L(K_{\mathcal{W}},\mathcal{V}) \text{ and } BT^{*}_{\mathbf{W},{\bf
w}}=0\}.
\end{align*}
In addition, by \cite[Lemma 4]{Christensen-Eldar (2006)},
\begin{align*}
\mathfrak{L}_{T^{*}_{\mathbf{W},{\bf
w}}}^{\mathcal{V},\mathcal{W}^{\perp}}&=\{B(T^{*}_{\mathbf{W},{\bf
w}}B)^{\dagger}:B\in L(\mathcal{H},K_{\mathcal{W}}) \text{ and }
{\rm
Im}(B)=\mathcal{V}\}\\&=\{\pi_{\mathcal{V},\mathcal{W}^{\perp}}(BT^{*}_{\mathbf{W},{\bf
w}})^{\dagger}B:B\in L(K_{\mathcal{W}},\mathcal{H}) \text{, }
\overline{{\rm Im}(B)}={\rm Im}(B) \text{ and } K_{\mathcal{W}}={\rm
Ker}(H) \oplus {\rm Im}(T^{*}_{\mathbf{W},{\bf w}})\}.
\end{align*}
\end{rem}

As a consequence of Theorem~\ref{T V,v dual fusion frame sii
Vi=ApiWj} we obtain the following alternative characterization of
component preserving oblique dual fusion frames.

\begin{thm}
Let $\mathcal{W}$ and $\mathcal{V}$ be finite-dimensional subspaces
of the Hilbert space $\mathcal{H}$ such that
$\mathcal{H}=\mathcal{V}\oplus \mathcal{W}^{\perp}$. Assume
$(\mathbf{W},{\bf w})$ is a fusion frame for $\mathcal{W}$. Then the
$Q$-component preserving oblique dual fusion frames of
$(\mathbf{W},{\bf w})$ on $\mathcal{V}$ are the families
\begin{equation}\label{expr1odff}
\{(V_i,v_{i})\}_{i\in
I}=\{(\pi_{\mathcal{V},{\mathcal{W}}^\bot}(HT^{*}_{{\bf W},{\bf
w}})^{\dagger}Z_i,v_i)\}_{i\in I},
\end{equation}
where $(\mathbf{Z},{\bf w})$ is a fusion frame sequence that
satisfies $HT^{*}_{{\bf W},{\bf w}}=T_{{\bf Z},{\bf
w}}(HM_iT^{*}_{{\bf W},{\bf w}})_{i \in I},$ for some $H \in
L(K_{\mathcal{W}}, \mathcal{H})$ with $\overline{{\rm Im}(H)}={\rm
Im}(H)$ and ${\rm Ker}(H)\oplus {\rm Im}(T^{*}_{{\bf W},{\bf
w}})=K_{\mathcal{W}}$ and
$Q=Q_{\pi_{\mathcal{V},{\mathcal{W}}^\bot}(HT^{*}_{{\bf W},{\bf
w}})^{\dagger}H,{\bf v}}$. Also,

\centerline{$\{(V_i,v_{i})\}_{i\in I}=\{(B(T^*_{{\bf W},{\bf
w}}B)^{\dagger}\pi_{ {\rm Im}(T^{*}_{{\bf W},{\bf w}}),\mathcal{S}}
 M_{i}K_{\mathcal{W}}, v_i)\}_{i\in I},$}

\noindent where $B\in L(K_{\mathcal{W}},\mathcal{H})$ is such that
${\rm Im}(B)=\mathcal{V}$, and $\mathcal{S}$ is a subspace of
$K_{\mathcal{W}}$ such that $K_{\mathcal{W}}={\rm Im}(T^{*}_{{\bf
W},{\bf w}})\oplus\mathcal{S}.$
\end{thm}

\begin{proof}
By Theorem~\ref{T V,v dual fusion frame sii Vi=ApiWj} and
Remark~\ref{R conjunto inversas oblicuas de TWw*}, the $Q$-oblique
component preserving oblique dual fusion frames of $(\mathbf{W},{\bf
w})$ on $\mathcal{V}$ are

\centerline{$\{(V_i,v_{i})\}_{i\in
I}=\{(\pi_{\mathcal{V},{\mathcal{W}}^\bot}(HT^{*}_{{\bf W},{\bf
w}})^{\dagger}HM_i(K_{\mathcal{W}}),v_i)\}_{i\in I},$}

\noindent where $H\in L(K_{\mathcal{W}},\mathcal{H})$ is such that
$\overline{{\rm Im}(H)}={\rm Im}(H)$ and $K_{\mathcal{W}}={\rm
Ker}(H)\oplus {\rm Im}(T^{*}_{{\bf W},{\bf w}}).$ Taking
$Z_i=HM_i(K_{\mathcal{W}})$ we get (\ref{expr1odff}).

From Remark~\ref{R conjunto inversas oblicuas de TWw*},
$\pi_{\mathcal{V},{\mathcal{W}}^\bot}=RT^{*}_{{\bf W},{\bf w}},$
where $R=B(T^*_{{\bf W},{\bf w}}B)^{\dagger}$ with $B\in
L(\mathcal{H},K_{\mathcal{W}})$ and ${\rm Im}(B)=\mathcal{V}$. So we
can write, by (\ref{expr1odff}),

\centerline{$V_i=RT^{*}_{{\bf W},{\bf w}}(HT^{*}_{{\bf W},{\bf
w}})^{\dagger}HM_i(K_{\mathcal{W}})=R\pi_{ {\rm Im}(T^{*}_{{\bf
W},{\bf w}}),\mathcal{S}}M_i(K_{\mathcal{W}}),$}

\noindent where $\mathcal{S}={\rm Ker}(H).$
\end{proof}

We will show now that we can construct component preserving oblique
dual fusion frame systems from a given fusion frame for a closed
subspace of $\mathcal{H}$ via local dual frames and an oblique left
inverse of its analysis operator.

\begin{prop}\label{P A TW T fusion frame system dual}
Let $\mathcal{V}$ and $\mathcal{W}$ be two closed subspaces of
$\mathcal{H}$ such that $\mathcal{H}=\mathcal{V} \oplus
\mathcal{W}^{\perp} $. Let $(\mathbf{W}, {\bf w})$ be a fusion frame
for $\mathcal{W}$, $A\in \mathfrak{L}_{T^{*}_{\mathbf{W},{\bf
w}}}^{\mathcal{V},\mathcal{W}^{\perp}}$ and ${\bf v}$ be a
collection of weights such that $\inf_{i \in I}v_i>0$. For each
$i\in I$ let $\{f_{i,l}\}_{l\in L_i}$ and $\{\tilde{f}_{i,l}\}_{l\in
L_i}$ be dual frames for $W_i$, $\beta_i$ upper frame bound of
$\{f_{i,l}\}_{l\in L_i}$ for each $i \in I$ such that $\sup_{i \in
I}\beta_i<\infty$, $\tilde{\alpha}_i$ and $\tilde{\beta}_i$ frame
bounds of $\{\tilde{f}_{i,l}\}_{l\in L_i}$ for each $i \in I$ such
that $\sup_{i \in I}\widetilde{\beta}_i<\infty$, $\mathcal{G}_i=\{
\frac{1}{v_{i}}A(\chi_{i}(j)\tilde{f}_{i,l})_{j \in I}\}_{l\in L_i}$
and $V_{i}=\overline{{\rm span}}\mathcal{G}_i$. Then
\begin{enumerate}
 \item  $\mathcal{G}_i$ is a frame for $V_{i}$ with frame
bounds $\|A^{\dagger}\|^{-2}\frac{\widetilde{\alpha}_i}{v_{i}^{2}}$
and $\|A\|^{2}\frac{\widetilde{\beta}_i}{v_{i}^{2}}.$
 \item $(\mathbf{V},\mathbf{v},\mathcal{G})$ is a component preserving $Q_{A,\mathbf{v}}$-oblique dual fusion frame system of $(\mathbf{W}, {\bf w},
 \mathcal{F})$ on $\mathcal{V}$.
\end{enumerate}
\end{prop}

\begin{proof}By
\cite[Proposition 5.3.1]{Christensen (2016)}, (1) holds.

Let $g \in \mathcal{H}$. We have $\sum_{i \in I}\sum_{l\in
L_{i}}|\langle g,A(\chi_{i}(j)\tilde{f}_{i,l})_{j \in
I}\rangle|^{2}=\sum_{i \in I}\sum_{l\in L_{i}}|\langle
A^{*}g,(\chi_{i}(j)\tilde{f}_{i,l})_{j \in I}\rangle|^{2}\leq\sum_{i
\in I}\widetilde{\beta}_i\|M_{i}A^{*}g\|^{2}\leq\sup_{i \in
I}\widetilde{\beta}_i\sum_{i \in
I}\|M_{i}A^{*}g\|^{2}=\|A^{*}g\|^{2}\sup_{i \in
I}\widetilde{\beta}_i\leq\|g\|^{2}\|A\|^{2}\sup_{i \in
I}\widetilde{\beta}_i$. As a consequence of Theorem~\ref{T wF marco
sii WwF fusion frame system}, $(\mathbf{V},\mathbf{v},\mathcal{G})$
is a Bessel fusion system for $\mathcal{H}$ with upper frame bound
$\|A\|^{2}\sup_{i \in I}\widetilde{\beta}_i$.

If $(h_i)_{i\in I} \in K_{\mathcal{W}}$, then $Q_{A,\mathbf{v}}$ is
a well defined bounded operator with
$||Q_{A,\mathbf{v}}||\leq\frac{||A||^{2}}{\inf_{i \in I}v_{i}^{2}}$
and $Q_{A,\mathbf{v}}(h_i)_{i\in I}=(\frac{1}{v_{i}}AM_{i}(h_{j})_{j
\in I})_{i\in I}=( \frac{1}{v_{i}}AM_{i}(\sum_{l\in
L_j}<h_j,f_j^l>\tilde{f}^l_j)_{j \in I})_{i\in I}=(\sum_{l\in
L_i}<h_i,f_{i,l}> \frac{1}{v_{i}}A(\chi_{i}(j)\tilde{f}_{i,l})_{j
\in I})_{i\in I}=C_{\mathcal{G}}C_{\mathcal{F}}^{*}(h_i)_{i\in I}$.
Hence (3) follows from (1), (2) and Lemma~\ref{L Vi=ApiWj entonces
V,v dual fusion frame}.\end{proof}

The next proposition presents a way to construct component
preserving oblique dual fusion frame systems from a given frame for
a subspace, using an oblique left inverse of its analysis operator.

\begin{prop}\label{P A TF T fusion frame system dual}
Let $\mathcal{V}$ and $\mathcal{W}$ be two closed subspaces of
$\mathcal{H}$ such that $\mathcal{H}=\mathcal{V} \oplus
\mathcal{W}^{\perp} $. Let ${\bf w}$ and ${\bf v}$ be two
collections of weights such that $\inf_{i \in I}v_i>0$. Let ${\bf
w}\mathcal{F}$ be a frame for $\mathcal{W}$ with local upper frame
bound $\beta_i$ for each $i \in I$ such that $\sup_{i \in
I}\beta_i<\infty$, $A \in \mathfrak{L}_{T_{{\bf
w}\mathcal{F}}^{*}}^{\mathcal{V},\mathcal{W}^{\perp}}$ and
$\{\{e_{i,l}\}_{l\in L_i}\}_{i\in I}$ be the standard basis for
$\oplus_{i \in I}\ell^2(L_{i})$. For each $i \in I$, set
$W_{i}=\overline{\text{span}}\{f_{i,l}\}_{l\in L_i}$ and
$V_{i}=\overline{\text{span}}\{\frac{1}{v_{i}}Ae_{i,l}\}_{l\in
L_i}$. Let $\mathcal{G}=\{\{\frac{1}{v_{i}}Ae_{i,l}\}_{l\in
L_i}\}_{i \in I}.$ Then
\begin{enumerate}
 \item $\{\frac{1}{v_{i}}Ae_{i,l}\}_{l\in L_i}$ is a frame for $V_{i}$ with frame
bounds $\frac{\|A^{\dagger}\|^{-2}}{v_{i}^{2}}$ and
$\frac{\|A\|^{2}}{v_{i}^{2}}$.
 \item $(\mathbf{V}, {\bf v}, \mathcal{G})$ is an oblique dual fusion frame system of $(\mathbf{W}, {\bf w},
 \mathcal{F})$ on $\mathcal{V}$.
\end{enumerate}
\end{prop}

\begin{proof}
Part (1) is a consecuence of \cite[Proposition 5.3.1]{Christensen
(2016)}.

If $g \in \mathcal{H}$, $\sum_{i \in I}\sum_{l\in L_{i}}|\langle
g,Ae_{i,l}\rangle|^{2}=\sum_{i \in I}\sum_{l\in L_{i}}|\langle
A^{*}g,e_{i,l}\rangle|^{2}=\|A^{*}g\|^{2}\leq\|A\|^{2}\|g\|^{2}$. By
Theorem~\ref{T wF marco sii WwF fusion frame system},
$(\mathbf{V},\mathbf{v},\mathcal{G})$ is a Bessel fusion system for
$\mathcal{H}$ with upper frame bound $\|A\|^{2}$.

By \cite[Lemma B.1]{Christensen-Eldar (2004)}, ${\bf v}\mathcal{G}$
is an oblique dual frame of ${\bf w}\mathcal{F}$ on $\mathcal{V}$.
So, part (2) follows from Theorem~\ref{T dual fusion frame
systems}.\end{proof}

\section{The canonical oblique dual fusion frame} \label{Ej dual canonico}

Let $(\mathbf{W},{\bf w})$ be a fusion frame for $\mathcal{W}$. Let
$A=\pi_{\mathcal{V},\mathcal{W}^{\perp}}S_{\mathbf{W},{\bf
w}}^{\dagger}T_{\mathbf{W},{\bf w}} \in
\mathfrak{L}_{T^{*}_{\mathbf{W},{\bf
w}}}^{\mathcal{V},\mathcal{W}^{\perp}}$ and ${\bf v}$ be a sequence
of weights such that
$(\pi_{\mathcal{V},\mathcal{W}^{\perp}}S_{\mathbf{W},{\bf
w}}^{\dagger}\mathbf{W},{\bf v})$ is a Bessel fusion sequence for
$\mathcal{V}$. Assume that $Q_{A,{\bf v}} : K_{\mathcal{W}}
\rightarrow \oplus_{i \in I}
\pi_{\mathcal{V},\mathcal{W}^{\perp}}S_{\mathbf{W},{\bf
w}}^{\dagger}W_i$ given by $Q_{A,{\bf v}}(f_i)_{i \in
I}=(\frac{w_{i}}{v_{i}}\pi_{\mathcal{V},\mathcal{W}^{\perp}}S_{\mathbf{W},{\bf
w}}^{\dagger}f_{i})_{i \in I}$ is a well defined bounded operator
(see \cite[Remark 3.6]{Heineken-Morillas-Benavente-Zakowicz
(2014)}). By Lemma~\ref{L Vi=ApiWj entonces V,v dual fusion frame},
$(\pi_{\mathcal{V},\mathcal{W}^{\perp}}S_{\mathbf{W},{\bf
w}}^{\dagger}\mathbf{W},{\bf v})$ is a $Q_{A,{\bf v}}$-component
preserving dual of $(\mathbf{W},{\bf w})$ on $\mathcal{V}$. In
particular we can take $\mathbf{v}=\mathbf{w}$. In fact,
$(S_{\mathbf{W},{\bf w}}^{\dagger}\mathbf{W},{\bf w})$ is a
$Q_{S_{\mathbf{W},{\bf w}}^{\dagger}T_{\mathbf{W},{\bf w}},{\bf
w}}$-component preserving dual of $(\mathbf{W},{\bf w})$ (see
\cite[Example 3.7]{Heineken-Morillas-Benavente-Zakowicz (2014)}).
Then, by Corollary~\ref{C dual entonces oblique dual},
$(\pi_{\mathcal{V},\mathcal{W}^{\perp}}S_{\mathbf{W},{\bf
w}}^{\dagger}\mathbf{W},{\bf w})$ is a $Q_{A,{\bf w}}$-oblique
component preserving dual of $(\mathbf{W},{\bf w})$ on
$\mathcal{V}$.

Given $\mathbf{v}$ we will refer to this dual as the {\em canonical
oblique dual with weights ${\bf v}$} and to

\centerline{$Q_{\pi_{\mathcal{V},\mathcal{W}^{\perp}}S_{\mathbf{W},{\bf
w}}^{\dagger}T_{\mathbf{W},{\bf w}},{\bf
v}}^{*}T_{\pi_{\mathcal{V},\mathcal{W}^{\perp}}S_{\mathbf{W},{\bf
w}}^{\dagger}\mathbf{W},{\bf v}}^{*}f=T_{\mathbf{W},{\bf
w}}^{*}S_{\mathbf{W},{\bf
w}}^{\dagger}\pi_{\mathcal{W},\mathcal{V}^{\perp}}f$}

\noindent as the {\em oblique fusion frame coefficients} of $f \in
\mathcal{H}$ with respect to $(\mathbf{W},{\bf w})$ on
$\mathcal{V}$. We note that if $\mathcal{V}=\mathcal{W}$, the
canonical oblique dual fusion frames reduce to the canonical dual
fusion frames as defined in
\cite{Heineken-Morillas-Benavente-Zakowicz (2014), Heineken-Morillas
(2014)}.

Furthermore, if in Definition~\ref{D oblique dual fusion frame
system} $({\bf V},{\bf v})$ is a canonical oblique dual fusion frame
of $(\mathbf{W},{\bf w})$  we say that $({\bf V},{\bf v},
\mathcal{G})$ is a \emph{canonical oblique dual fusion frame system}
of $(\mathbf{W},{\bf w}, \mathcal{F})$.

The theorem below follows from Theorem~\ref{T V,v dual fusion frame
sii Vi=ApiWj}. It gives a characterizations of canonical oblique
dual fusion frames and can be proved in a similar way as
\cite[Theorem 2]{Christensen-Eldar (2006)}.
\begin{thm}
Let $\mathcal{V}$ and $\mathcal{W}$ be two finite-dimensional
subspaces of $\mathcal{H}$ such that $\mathcal{H}=\mathcal{V} \oplus
\mathcal{W}^{\perp} $. Assume $(\mathbf{W},{\bf w})$ is a fusion
frame for a subspace $\mathcal{W}\subseteq \mathcal{H}$. Then the
canonical $Q$-oblique dual fusion frames of $(\mathbf{W},{\bf w})$
on $\mathcal{V}$ are

\centerline{$\{(V_i,v_{i})\}_{i\in
I}=\{(\pi_{\mathcal{V},{\mathcal{W}}^\bot}(HT^{*}_{{\bf W},{\bf
w}})^{\dagger}Z_i,v_i)\}_{i\in I},$}

\noindent where $(\mathbf{Z},{\bf w})$ is a fusion frame sequence
that satisfies $HT^{*}_{{\bf W},{\bf w}}=T_{{\bf Z},{\bf
w}}(HM_iT^{*}_{{\bf W},{\bf w}})_{i \in I},$ for some $H \in
L(K_{\mathcal{W}}, \mathcal{H})$ with ${\rm Ker}(H)={\rm
Ker}(T_{{\bf W},{\bf w}})$ and
$Q=Q_{\pi_{\mathcal{V},{\mathcal{W}}^\bot}(HT^{*}_{{\bf W},{\bf
w}})^{\dagger}H,{\bf v}}$. Also

\centerline{$\{(V_i,v_{i})\}_{i\in I}=\{(B(T^{*}_{{\bf W},{\bf
w}}B)^{\dagger}M_i(K_{\mathcal{W}}),v_i)\}_{i\in I},$}

\noindent where $B\in L(K_{\mathcal{W}},\mathcal{H})$ is such that
${\rm Im}(B)=\mathcal{V}.$
\end{thm}
The following lemma implies that oblique fusion frame coefficients
are those which have minimal norm among all other coefficients.
\begin{lem}
Let $\mathcal{W}$ and $\mathcal{V}$ be two closed subspaces of
$\mathcal{H}$ such that
$\mathcal{H}=\mathcal{V}\oplus\mathcal{W}^{\perp}$. Let
$(\mathbf{W},{\bf w})$ be a fusion frame for $\mathcal{W}$ and $f
\in \mathcal{H}$. For all $(f_{i})_{i \in I} \in K_{\mathcal{W}}$
satisfying $T_{\mathbf{W},{\bf w}}(f_{i})_{i \in
I}=\pi_{\mathcal{W},\mathcal{V}^{\perp}}f$, we have

\centerline{$||(f_{i})_{i \in I}||^{2}=||T_{\mathbf{W},{\bf
w}}^{*}S_{\mathbf{W},{\bf
w}}^{\dagger}\pi_{\mathcal{W},\mathcal{V}^{\perp}}f||^{2}+||(f_{i})_{i
\in I}-T_{\mathbf{W},{\bf w}}^{*}S_{\mathbf{W},{\bf
w}}^{\dagger}\pi_{\mathcal{W},\mathcal{V}^{\perp}}f||^{2}.$}
\end{lem}
\begin{proof}
Suppose that $(f_{i})_{i \in I} \in K_{\mathcal{W}}$ satisfies
$T_{\mathbf{W},{\bf w}}(f_{i})_{i \in
I}=\pi_{\mathcal{W},\mathcal{V}^{\perp}}f$. Then

\centerline{$(f_{i})_{i \in I}-T_{\mathbf{W},{\bf
w}}^{*}S_{\mathbf{W},{\bf
w}}^{\dagger}\pi_{\mathcal{W},\mathcal{V}^{\perp}}f \in {\rm
Ker}(T_{\mathbf{W},{\bf w}})={\rm Im}(T_{\mathbf{W},{\bf
w}}^{*})^{\perp}$.}

\noindent Since $T_{\mathbf{W},{\bf w}}^{*}S_{\mathbf{W},{\bf
w}}^{\dagger}\pi_{\mathcal{W},\mathcal{V}^{\perp}}f \in {\rm
Im}(T_{\mathbf{W},{\bf w}}^{*})$, the conclusion follows.
\end{proof}

\begin{rem}\label{R dual trivial}
Let $\mathcal{W}$ and $\mathcal{V}$ be two closed subspaces of
$\mathcal{H}$ such that
$\mathcal{H}=\mathcal{V}\oplus\mathcal{W}^{\perp}$. Let
$(\mathbf{W},{\bf w})$ be a fusion frame for $\mathcal{W}$ and
$({\bf V},{\bf v})$  be a fusion frame for $\mathcal{V}$.

Since $\mathcal{H}=\mathcal{V}\oplus\mathcal{W}^{\perp}$, the
operator $T_{\mathbf{W},{\bf w}}^{*}T_{\mathbf{V},{\bf v}} :
K_{\mathcal{V}} \rightarrow K_{\mathcal{W}}$ given by
$T_{\mathbf{W},{\bf w}}^{*}T_{\mathbf{V},{\bf v}}(g_{i})_{i \in
I}=\sum_{i \in I}v_{i}T_{\mathbf{W},{\bf w}}^{*}g_{i}=(w_{l}\sum_{i
\in I}v_{i}\pi_{W_{l}}g_{i})_{l \in I}$, satisfies ${\rm
Ker}(T_{\mathbf{W},{\bf w}}^{*}T_{\mathbf{V},{\bf v}})={\rm
Ker}(T_{\mathbf{V},{\bf v}})$ and ${\rm Im}(T_{\mathbf{W},{\bf
w}}^{*}T_{\mathbf{V},{\bf v}})={\rm Im}(T_{\mathbf{W},{\bf
w}}^{*})$.

If $f \in \mathcal{W}^{\perp}={\rm Ker}(T_{\mathbf{W},{\bf
w}}^{*})$, $T_{\mathbf{V},{\bf v}}(T_{\mathbf{W},{\bf
w}}^{*}T_{\mathbf{V},{\bf v}})^{\dagger}T_{\mathbf{W},{\bf
w}}^{*}f=0$. If $g \in \mathcal{V}={\rm Im}(T_{\mathbf{V},{\bf
v}})$, there exists $(g_{i})_{i \in I} \in K_{\mathcal{V}}$ such
that $g=T_{\mathbf{V},{\bf v}}(g_{i})_{i \in I}$. Then
$T_{\mathbf{V},{\bf v}}(T_{\mathbf{W},{\bf w}}^{*}T_{\mathbf{V},{\bf
v}})^{\dagger}T_{\mathbf{W},{\bf w}}^{*}g=T_{\mathbf{V},{\bf
v}}\pi_{{\rm Ker}(T_{\mathbf{V},{\bf v}})^{\perp}}(g_{i})_{i \in
I}=T_{\mathbf{V},{\bf v}}(g_{i})_{i \in I}=g$. Thus
$T_{\mathbf{V},{\bf v}}(T_{\mathbf{W},{\bf w}}^{*}T_{\mathbf{V},{\bf
v}})^{\dagger}T_{\mathbf{W},{\bf
w}}^{*}=\pi_{\mathcal{V},\mathcal{W}^{\perp}}$ and consequently,
$({\bf V},{\bf v})$ is a $(T_{\mathbf{W},{\bf
w}}^{*}T_{\mathbf{V},{\bf v}})^{\dagger}$-oblique dual fusion frame
of $(\mathbf{W},{\bf w})$ on $\mathcal{V}$.

This shows that given $(\mathbf{W},{\bf w})$ a fusion frame for
$\mathcal{W}$ and $({\bf V},{\bf v})$  a fusion frame for
$\mathcal{V}$ we can always do the analysis with one of them and the
synthesis with the other leading to a consistent reconstruction.
Note that this happens in the general framework of Definition~\ref{D
oblique fusion frame dual} where we do not impose any additional
condition on $Q$.

We note that the component preserving oblique duals associate with
$T_{\mathbf{V},{\bf v}}(T_{\mathbf{W},{\bf w}}^{*}T_{\mathbf{V},{\bf
v}})^{\dagger} \in \mathfrak{L}_{T^{*}_{\mathbf{W},{\bf
w}}}^{\mathcal{V},\mathcal{W}^{\perp}}$ obtained applying
Lemma~\ref{L Vi=ApiWj entonces V,v dual fusion frame} are the
canonical ones. To see this we will prove that $T_{\mathbf{V},{\bf
v}}(T_{\mathbf{W},{\bf w}}^{*}T_{\mathbf{V},{\bf
v}})^{\dagger}=\pi_{\mathcal{V},\mathcal{W}^{\perp}}S_{\mathbf{W},{\bf
w}}^{\dagger}T_{\mathbf{W},{\bf w}}$. We have on one hand

\centerline{$T_{\mathbf{V},{\bf v}}(T_{\mathbf{W},{\bf
w}}^{*}T_{\mathbf{V},{\bf v}})^{\dagger}T_{\mathbf{W},{\bf
w}}^{*}=\pi_{\mathcal{V},\mathcal{W}^{\perp}}S_{\mathbf{W},{\bf
w}}^{\dagger}T_{\mathbf{W},{\bf w}}T_{\mathbf{W},{\bf
w}}^{*}=\pi_{\mathcal{V},\mathcal{W}^{\perp}}$.}

\noindent Let now $(g_{i})_{i \in I} \in {\rm Im}(T_{\mathbf{W},{\bf
w}}^{*})^{\perp}={\rm Ker}(T_{\mathbf{W},{\bf w}})={\rm
Ker}((T_{\mathbf{W},{\bf w}}^{*}T_{\mathbf{V},{\bf v}})^{\dagger})$.
Then $T_{\mathbf{V},{\bf v}}(T_{\mathbf{W},{\bf
w}}^{*}T_{\mathbf{V},{\bf v}})^{\dagger}(g_{i})_{i \in
I}=\pi_{\mathcal{V},\mathcal{W}^{\perp}}S_{\mathbf{W},{\bf
w}}^{\dagger}T_{\mathbf{W},{\bf w}}(g_{i})_{i \in I}=0$.
\end{rem}

\subsection{Existence of non-canonical oblique dual fusion frames}\label{Existence of non-canonical}

A Bessel fusion sequence $(\mathbf{W},{\bf w})$ is a Riesz fusion
basis for $\mathcal{W}$ if and only if $T_{\mathbf{W},{\bf w}}$ is
injective, or equivalently, $S_{\mathbf{W},{\bf
w}}^{\dagger}T_{\mathbf{W},{\bf w}}$ is the unique element in
$\mathfrak{L}_{T_{\mathbf{W},{\bf w}}^{*}}^{\mathcal{W}}$. In this
case, by Proposition~\ref{P relacion inversa inversa oblicua},
$\pi_{\mathcal{V},\mathcal{W}^{\perp}}S_{\mathbf{W},{\bf
w}}^{\dagger}T_{\mathbf{W},{\bf w}}$ is the unique element in
$\mathfrak{L}_{T_{\mathbf{W},{\bf
w}}^{*}}^{\mathcal{V},\mathcal{W}^{\perp}}$. So, by Lemma~\ref{L V,v
dual fusion frame entonces Vi=ApiWj}, if $(\mathbf{W},{\bf w})$ is a
Riesz fusion basis for $\mathcal{W}$ the only component preserving
duals of $(\mathbf{W},{\bf w})$ are
$(\pi_{\mathcal{V},\mathcal{W}^{\perp}}S_{\mathbf{W},{\bf
w}}^{\dagger}\mathbf{W},{\bf v})$. It is easy to see that if
$Q_{\pi_{\mathcal{V},\mathcal{W}^{\perp}}S_{\mathbf{W},{\bf
w}}^{\dagger}T_{\mathbf{W},{\bf w}},{\bf v}}$ is well defined and
bounded for the weights $\mathbf{v}$, this component preserving
oblique dual coincides with the canonical one with weights
$\mathbf{v}$, i.e. the operator $Q$ for this dual is
$Q_{\pi_{\mathcal{V},\mathcal{W}^{\perp}}S_{\mathbf{W},{\bf
w}}^{\dagger}T_{\mathbf{W},{\bf w}},{\bf v}}$. We also have:
\begin{prop}\label{DualRFB} Let $\mathcal{W}$ and $\mathcal{V}$ be
two closed subspaces of $\mathcal{H}$ such that
$\mathcal{H}=\mathcal{V}\oplus\mathcal{W}^{\perp}$. Let
$(\mathbf{W},{\bf w})$ be a Riesz fusion basis for $\mathcal{W}$ and
${\bf v}$ a family of weights. The following assertions hold:
\begin{enumerate}

\item $(\pi_{\mathcal{V},\mathcal{W}^{\perp}}S_{\mathbf{W},{\bf
w}}^{\dagger}\mathbf{W},{\bf w})$ is a Riesz fusion basis for
$\mathcal{V}$.

\item If $({\bf V},{\bf v})$ is a Riesz fusion basis for $\mathcal{V}$, then
$T_{\mathbf{V},{\bf v}}^{*}T_{\mathbf{W},{\bf w}}$ is invertible.

  \item \label{MC} Let $({\bf V},{\bf v})$ be a block-diagonal oblique dual fusion frame of  $(\mathbf{W},{\bf w})$ on $\mathcal{V}$. Then, for each $i \in I$, $\pi_{\mathcal{V},\mathcal{W}^{\perp}}S_{\mathbf{W},{\bf
w}}^{\dagger}W_{i} \subseteq
  V_{i}$.
\item \label{CFRB} If $({\bf V},{\bf v})$ is a Riesz fusion basis for $\mathcal{V}$ which is a block-diagonal oblique dual fusion frame of $(\mathbf{W},{\bf w})$ on $\mathcal{V}$, then $V_i=\pi_{\mathcal{V},\mathcal{W}^{\perp}}S_{\mathbf{W},{\bf
w}}^{\dagger}W_{i}$ for $i \in I$

     \end{enumerate}
\end{prop}
\begin{proof}

(1) Let $(f_{i})_{i \in I} \in K_{\mathcal{W}}$. Since
$(\mathbf{W},{\bf w})$ is a Riesz fusion basis for $\mathcal{W}$,
there exists $f \in \mathcal{H}$ such that $(f_{i})_{i \in
I}=T_{\mathbf{W},{\bf w}}^{*}f$. Then \begin{align*}
T_{\pi_{\mathcal{V},\mathcal{W}^{\perp}}S_{\mathbf{W},{\bf
w}}^{\dagger}\mathbf{W},{\bf
w}}(\pi_{\mathcal{V},\mathcal{W}^{\perp}}S_{\mathbf{W},{\bf
w}}^{\dagger}f_{i})_{i \in
I}&=\pi_{\mathcal{V},\mathcal{W}^{\perp}}S_{\mathbf{W},{\bf
w}}^{\dagger}T_{\mathbf{W},{\bf w}}(f_{i})_{i \in
I}\\&=\pi_{\mathcal{V},\mathcal{W}^{\perp}}S_{\mathbf{W},{\bf
w}}^{\dagger}T_{\mathbf{W},{\bf w}}T_{\mathbf{W},{\bf
w}}^{*}f=\pi_{\mathcal{V},\mathcal{W}^{\perp}}\pi_{\mathcal{W}}f=\pi_{\mathcal{V},\mathcal{W}^{\perp}}f.
\end{align*} Thus
$T_{\pi_{\mathcal{V},\mathcal{W}^{\perp}}S_{\mathbf{W},{\bf
w}}^{\dagger}\mathbf{W},{\bf
w}}(\pi_{\mathcal{V},\mathcal{W}^{\perp}}S_{\mathbf{W},{\bf
w}}^{\dagger}f_{i})_{i \in I}=0$ if and only if $f \in
\mathcal{W}^{\perp}={\rm Ker}(T_{\mathbf{W},{\bf w}}^{*})$, that is,
$(f_{i})_{i \in I}=0$. It follows that
$T_{\pi_{\mathcal{V},\mathcal{W}^{\perp}}S_{\mathbf{W},{\bf
w}}^{\dagger}\mathbf{W},{\bf w}}$ is injective, or equivalently,
$(\pi_{\mathcal{V},\mathcal{W}^{\perp}}S_{\mathbf{W},{\bf
w}}^{\dagger}\mathbf{W},{\bf w})$ is a Riesz fusion basis for
$\mathcal{V}$.

(2) Let $(f_{i})_{i \in I} \in K_{\mathcal{W}}$ such that
$T_{\mathbf{V},{\bf v}}^{*}T_{\mathbf{W},{\bf w}}(f_{i})_{i \in
I}=0$. So, $T_{\mathbf{W},{\bf w}}(f_{i})_{i \in I} \in {\rm
Ker}(T_{\mathbf{V},{\bf v}}^{*}) \cap {\rm Im}(T_{\mathbf{W},{\bf
w}})=\mathcal{V}^{\perp}\cap\mathcal{W}=\{0\}$. Since
$(\mathbf{W},{\bf w})$ is a Riesz fusion basis for $\mathcal{W}$,
this implies that $(f_{i})_{i \in I}=0$. Therefore,
$T_{\mathbf{V},{\bf v}}^{*}T_{\mathbf{W},{\bf w}}$ is injective. In
the same manner it results that $(T_{\mathbf{V},{\bf
v}}^{*}T_{\mathbf{W},{\bf w}})^{*}=T_{\mathbf{W},{\bf
w}}^{*}T_{\mathbf{V},{\bf v}}$ is injective. Consequently,
$T_{\mathbf{V},{\bf v}}^{*}T_{\mathbf{W},{\bf w}}$ is bijective.

(\ref{MC}) Let $(f_{i})_{i \in I} \in K_{\mathcal{W}}$ and $i \in I$
fixed. From $T_{{\bf V},{\bf v}}QT_{\mathbf{W},{\bf
w}}^{*}=\pi_{\mathcal{V},\mathcal{W}^{\perp}}$ and
$T_{\mathbf{W},{\bf w}}^{*}S_{\mathbf{W},{\bf
w}}^{\dagger}T_{\mathbf{W},{\bf w}}=I_{K_{\mathcal{W}}}$ (the last
equality holds since $(T_{\mathbf{W},{\bf
w}}^{*})_{|\mathcal{W}}:\mathcal{W}\rightarrow K_{\mathcal{W}}$ is
bijective and $S_{\mathbf{W},{\bf w}}^{\dagger}T_{\mathbf{W},{\bf
w}}T_{\mathbf{W},{\bf w}}^{*}=\pi_{\mathcal{W}}$), we obtain
\begin{align*}
\pi_{\mathcal{V},\mathcal{W}^{\perp}}S_{\mathbf{W},{\bf
w}}^{\dagger}f_{i}&=T_{{\bf V},{\bf v}}QT_{\mathbf{W},{\bf
w}}^{*}S_{\mathbf{W},{\bf w}}^{\dagger}T_{\mathbf{W},{\bf
w}}M_{i}(\chi_{i}(j)\frac{1}{w_{j}}f_{j})_{j \in I}\\&= T_{{\bf
V},{\bf v}}QM_{i}(\chi_{i}(j)\frac{1}{w_{j}}f_{j})_{j \in I}=T_{{\bf
V},{\bf v}}M_{i}Q(\chi_{i}(j)\frac{1}{w_{j}}f_{j})_{j \in I}\in
V_{i}.
\end{align*}
\noindent So
$\pi_{\mathcal{V},\mathcal{W}^{\perp}}S_{\mathbf{W},{\bf
w}}^{\dagger}W_{i} \subseteq V_{i}.$

(\ref{CFRB}) By (\ref{MC}),
$\pi_{\mathcal{V},\mathcal{W}^{\perp}}S_{\mathbf{W},{\bf
w}}^{\dagger}W_{i}\subseteq V_i$ for each $i \in I$. Suppose that
there exists $i_0\in I$ such that
$\pi_{\mathcal{V},\mathcal{W}^{\perp}}S_{\mathbf{W},{\bf
w}}^{\dagger}W_{i_0}\subset V_{ i_0}$. Set $\{0\} \neq U_{i_{0}}
\subset V_{ i_0}$ such that
$V_{i_{0}}=\pi_{\mathcal{V},\mathcal{W}^{\perp}}S_{\mathbf{W},{\bf
w}}^{\dagger}W_{i_0}\oplus U_{i_{0}}$ Take $0 \neq u_{i_{0}} \in U_{
i_0}$. By (1), $u_{i_{0}}=\sum_{i \in I}g_{i}$ where $g_{i} \in
\pi_{\mathcal{V},\mathcal{W}^{\perp}}S_{\mathbf{W},{\bf
w}}^{\dagger}W_{i}$ for each $i \in I$. Since
$\pi_{\mathcal{V},\mathcal{W}^{\perp}}S_{\mathbf{W},{\bf
w}}^{\dagger}W_{i}\subseteq V_i$ for each $i \in I$ and $({\bf
V},{\bf v})$ is a Riesz fusion basis for $\mathcal{V}$,
$u_{i_{0}}=g_{i_{0}}\in
\pi_{\mathcal{V},\mathcal{W}^{\perp}}S_{\mathbf{W},{\bf
w}}^{\dagger}W_{i_0} \cap U_{ i_0}=\{0\}$. This is absurd. Thus the
conclusion follows.\end{proof}

\begin{rem}
Let $\mathcal{W}$ and $\mathcal{V}$ be two closed subspaces of
$\mathcal{H}$. Let $(\mathbf{W},{\bf w})$ be a fusion Riesz basis
for $\mathcal{W}$ and $({\bf V},{\bf v})$ a fusion Riesz basis for
$\mathcal{V}$. If $T_{\mathbf{V},{\bf v}}^{*}T_{\mathbf{W},{\bf w}}$
is injective, then $\mathcal{W} \cap \mathcal{V}^{\perp}=\{0\}$. To
see this, consider $f \in \mathcal{W} \cap \mathcal{V}^{\perp}$.
Since $f \in \mathcal{W}={\rm Im}(T_{\mathbf{W},{\bf w}})$, there
exists $(f_{i})_{i \in I} \in K_{\mathcal{W}}$ such that
$f=T_{\mathbf{W},{\bf w}}(f_{i})_{i \in I}$. Since $f \in
\mathcal{V}^{\perp}={\rm Ker}(T_{\mathbf{V},{\bf v}}^{*})$,
$T_{\mathbf{V},{\bf v}}^{*}f=0$. Therefore, $T_{\mathbf{V},{\bf
v}}^{*}T_{\mathbf{W},{\bf w}}(f_{i})_{i \in I}=0$. Taking into
account that $T_{\mathbf{V},{\bf v}}^{*}T_{\mathbf{W},{\bf w}}$ is
injective, we deduce that $(f_{i})_{i \in I}=0$, and then $f=0$.
\end{rem}

To prove Proposition~\ref{P existencia no canonicos} about the
existence of non canonical oblique duals, we will need the following
Corollary which is a consequence of the next lemma that generalizes
Lemma 5.5.5 in \cite{Christensen (2016)}, and can be proved in a
similar way.

\begin{lem}\label{L Caracterizacion fusion frames}
Let $\mathcal{W}$ be a closed subspace of $\mathcal{H}$. A pair
$(\mathbf{W},{\bf w})$ is a fusion frame for $\mathcal{W}$ with
bounds $A,\,B$ if and only if the following conditions are
satisfied:
\begin{enumerate}
\item $(\mathbf{W},{\bf w})$ is complete in $\mathcal{W}$
\item The operator $T_{\mathbf{W},{\bf
w}}$ is well defined on $\mathcal{K}_{\mathcal{W}}$ and

\centerline{$A\|(f_i)_{i\in I}\|^2\leq \|T_{\mathbf{W},{\bf
w}}(f_i)_{i\in I}\|^2\leq B \|(f_i)_{i\in I}\|^2\,\,\,\forall
(f_i)_{i\in I}\in {\rm Ker}(T_{\mathbf{W},{\bf w}})^{\bot}$}
\end{enumerate}
\end{lem}

\begin{cor}\label{C FF o incomplete}
Let $\mathcal{W}$ be a closed subspace of $\mathcal{H}$ and
$(\mathbf{W},{\bf w})$ be a fusion frame for $\mathcal{W}$. Let
$(\mathbf{\widetilde{W}},{\bf w})$ be a sequence such that
$\widetilde{W}_i\subseteq W_i$ for all $i\in I.$ Then
$(\mathbf{\widetilde{W}},{\bf w})$ is either a fusion frame for
$\mathcal{W}$ or incomplete in $\mathcal{W}$.
\end{cor}

\begin{proof}
Assume that $(\mathbf{\widetilde{W}},{\bf w})$ is complete in
$\mathcal{W}$. Since
$\pi_{\widetilde{W}_{i}}=\pi_{\widetilde{W}_{i}}\pi_{W_{i}}$ It is
clear that $(\mathbf{\widetilde{W}},{\bf w})$ is a Bessel sequence,
so the operator $T_{\widetilde{\mathbf{W}},{\bf w}}$ is well defined
on $\mathcal{K}_{\widetilde{\mathcal{W}}}$.

Considering everything inside $\mathcal{K}_{{\mathcal{\widetilde
W}}}$ we can decompose

\centerline{${\rm Ker}(T_{{\mathbf{\widetilde{W}},{\bf
w}}})^{\bot}={\rm Ker}(T_{\mathbf{W},{\bf w}})^{\bot}\oplus ({\rm
Ker}(T_{\mathbf{W},{\bf w}})^{\bot})^{\bot}={\rm
Ker}(T_{\mathbf{W},{\bf w}})^{\bot}\oplus {\rm
Ker}(T_{\mathbf{W},{\bf w}}).$}

\noindent Hence, by Lemma~\ref{L Caracterizacion fusion frames},
\begin{equation}\label{desinf}
\|T_{\mathbf{\widetilde{W}},{\bf w}}(f_i)_{i\in I}\|^2\geq
A\|\pi_{{\rm Ker}(T_{\mathbf{W},{\bf w}})^{\bot}}(f_i)_{i\in
I}\|^2\,\,\,\forall (f_i)_{i\in I}\in {\rm
Ker}(T_{{\mathbf{\widetilde{W}},{\bf w}}})^{\bot}.
\end{equation}
Since $\text{span}\bigcup_{i\in I} \widetilde{W}_i \subseteq {\rm
Im}(T_{\widetilde{\mathbf{W}},{\bf w}}) \subseteq
\overline{\text{span}}\bigcup_{i\in I} \widetilde{W}_i =
\mathcal{W}$, we only have to prove that ${\rm
Im}(T_{{\mathbf{\widetilde W}},{\bf w}})$ is closed. Let $g\in
\overline{{\rm Im}(T_{{\mathbf{\widetilde W}},{\bf w}})}$. Then
there exists a sequence $(f_i)_{i\in I}^n \in {\rm
Ker}(T_{\mathbf{W},{\bf w}})^{\bot}$ such that
$T_{\mathbf{\widetilde{W}},{\bf w}}(f_i)_{i\in I}^n$ converges to
$g$. By (\ref{desinf}) $(f_i)_{i\in I}^n$ is a Cauchy sequence, so
it converges to some $(f_i)_{i\in I}$ in
$\mathcal{K}_{{\mathcal{\widetilde W}}}$ which satisfies, by
continuity, that $T_{\mathbf{\widetilde{W}},{\bf w}}(f_i)_{i\in
I}=g$.
\end{proof}

The next proposition shows that if $(\mathbf{W},{\bf w})$ is an
overcomplete fusion frame with non trivial subspaces there always
exist component preserving oblique dual fusion frames which are not
the canonical ones. For $\mathcal{V}=\mathcal{W}$ this result is a
generalization of \cite[Proposition 3.9]{Heineken-Morillas (2014)}
to the infinite dimensional case.

\begin{prop}\label{P existencia no canonicos}
Let $(\mathbf{W},{\bf w})$ be a fusion frame for a closed subspace
$\mathcal{W}\subseteq \mathcal{H}$ and let $\mathcal{V}$ be a closed
subspace such that
$\mathcal{H}=\mathcal{V}\oplus\mathcal{W}^{\perp}$. Let
$(\mathbf{W},{\bf w})$ be an overcomplete fusion frame for
$\mathcal{W}$ such that $W_i\neq \{0\}$  for every $i\in I$. Then
there exist component preserving oblique dual fusion frames
$(\mathbf{V},\mathbf{w})$ of $(\mathbf{W},{\bf w})$ different from
$(\pi_{\mathcal{V},\mathcal{W}^{\perp}}S_{\mathbf{W},{\bf
w}}^{\dagger}\mathbf{W},{\bf  w})$.
\end{prop}

\begin{proof} Since  $(\mathbf{W},{\bf w})$ is not a Riesz fusion basis, there exists $i_0\in
I$ such that $W_{i_0}\cap \overline{\text{span}}\{W_i:i\neq
i_0\}\neq \{0\}.$ Let us first prove that
$(\mathbf{\widetilde{W}},{\bf w})$ given by $\widetilde{W}_i=W_i$
for $i\neq i_0$ and
$\widetilde{W}_{i_0}=W_{i_0}\cap(W_{i_0}\cap\overline{\text{span}}\{W_i:i\neq
i_0\})^\bot$ is a fusion frame for $\mathcal{W}$.

Let us first see that $(\mathbf{\widetilde{W}},{\bf w})$ given by
$\widetilde{W}_i=W_i$ for $i\neq i_0$ and
$\widetilde{W}_{i_0}=W_{i_0}\cap(\overline{\text{span}}\{W_i:i\neq
i_0\}\cap W_{i_0})^\bot$ is a fusion frame for $\mathcal{W}$.

Let $f\in \mathcal{W}$. Then $f=\sum_{i\in I}f_i$, with $f_i\in W_i$
for all $i\in I.$ So, $f=\sum_{i\neq i_0}f_i+f_{i_0}=\sum_{i\neq
i_0}f_i+\pi_{\overline{\text{span}}\{W_i:i\neq i_0\}\cap
W_{i_0}}(f_{i_0})+\pi_{(\overline{\text{span}}\{W_i:i\neq i_0\}\cap
W_{i_0})^\bot}(f_{i_0})$.

But then $\pi_{(\overline{\text{span}}\{W_i:i\neq i_0\}\cap
W_{i_0})^\bot}(f_{i_0})\in W_{i_0}$, hence
$\pi_{(\overline{\text{span}}\{W_i:i\neq i_0\}\cap
W_{i_0})^\bot}(f_{i_0})\in \widetilde{W}_{i_0}$. So $f\in
\overline{\text{span}}(\bigcup_{i=1}^m\widetilde{W}_i)$. It follows
that $(\mathbf{\widetilde{W}},{\bf w})$ is complete. By
Corollary~\ref{C FF o incomplete}, $(\mathbf{\widetilde{W}},{\bf
w})$ is a fusion frame for $\mathcal{W}$.

Now define
$V_i=\pi_{\mathcal{V},\mathcal{W}^{\perp}}S_{\mathbf{\widetilde{W}},{\bf
w}}^{\dagger}(\widetilde{W}_{i})$ for $i\in I.$ Consider the
component preserving $\widetilde{Q}\in L(K_{\mathcal{W}},
K_{\mathcal{V}}),$ given by $\widetilde{Q}(f_i)_{i\in
I}=(\pi_{\mathcal{V},\mathcal{W}^{\perp}}S_{\mathbf{\widetilde{W}},{\bf
w}}^{\dagger}\pi_{\widetilde{W}_i}f_i)_{i\in I}.$

Let $f\in\mathcal{H}.$ Since
$\pi_{\widetilde{W}_{i_0}}\pi_{W_{i_0}}f=\pi_{\widetilde{W}_{i_0}}f,$
we obtain $T_{\mathbf{V},{\bf w}}\widetilde{Q}T_{\mathbf{W},{\bf
w}}^*f=\sum_{i \in I}w_i^2
\pi_{\mathcal{V},\mathcal{W}^{\perp}}S_{\mathbf{\widetilde{W}},{\bf
w}}^{\dagger}\pi_{\widetilde{W}_i}(f)=\pi_{\mathcal{V},\mathcal{W}^{\perp}}(f).$
This shows that $(\mathbf{V},\mathbf{w})$ is a component preserving
$\widetilde{Q}$-oblique dual fusion frame of
$(\mathbf{W},\mathbf{w})$ on $\mathcal{V}$.

Note that $\widetilde{W}_{i_0}\subseteq W_{i_0}.$ Assume that
$\widetilde{W}_{i_0}=W_{i_0}.$ Then $W_{i_0}\subseteq
(\overline{\text{span}}\{W_i:i\neq i_0\} \cap W_{i_0})^\bot$ which
is a contradiction since $W_{i_0}\cap
\overline{\text{span}}\{W_i:i\neq i_0\}\neq \{0\}.$  To see that $
\pi_{\mathcal{V},\mathcal{W}^{\perp}}S_{\mathbf{\widetilde{W}},{\bf
w}}^{\dagger}({\widetilde{W}_{i_0}})\neq
\pi_{\mathcal{V},\mathcal{W}^{\perp}}S_{\mathbf{W},{\bf
w}}^{\dagger}(W_{i_0})$, take $f\in S_{\mathbf{W},{\bf
w}}^{\dagger}(W_{i_0}\cap {\widetilde{W}_{i_0}}^\bot),\,\,f\neq 0$.
Assume, by contradiction, that
$\pi_{\mathcal{V},\mathcal{W}^{\perp}}f\in
\pi_{\mathcal{V},\mathcal{W}^{\perp}}S_{\mathbf{\widetilde{W}},{\bf
w}}^{\dagger}({\widetilde{W}_{i_0}})$. We have

\centerline{$\pi_{\mathcal{V},\mathcal{W}^{\perp}}(f)=\pi_{\mathcal{V},\mathcal{W}^{\perp}}S_{\mathbf{\widetilde{W}},{\bf
w}}^{\dagger}S_{\mathbf{\widetilde{W}},{\bf w}}(f)=
\pi_{\mathcal{V},\mathcal{W}^{\perp}}S_{\mathbf{\widetilde{W}},{\bf
w}}^{\dagger}(\sum_{i\in I,i\neq i_0} w_i^2\pi_{W_{i}}(f)+w_{i_0}^2
\pi_{{\widetilde {W}}_{i_0}}(f)).$}

\noindent Then
$\pi_{\mathcal{V},\mathcal{W}^{\perp}}S_{\mathbf{\widetilde{W}},{\bf
w}}^{\dagger}(\sum_{i\in I,i\neq i_0} w_i^2\pi_{W_{i}}(f))\in
\pi_{\mathcal{V},\mathcal{W}^{\perp}}S_{\mathbf{\widetilde{W}},{\bf
w}}^{\dagger}({\widetilde{W}_{i_0}})$, i.e. there exists $g\in
{\widetilde{W}_{i_0}}$ such that
$\pi_{\mathcal{V},\mathcal{W}^{\perp}}S_{\mathbf{\widetilde{W}},{\bf
w}}^{\dagger}(\sum_{i\in I,i\neq i_0}
w_i^2\pi_{W_{i}}(f))=\pi_{\mathcal{V},\mathcal{W}^{\perp}}S_{\mathbf{\widetilde{W}},{\bf
w}}^{\dagger}(g).$ Hence $S_{\mathbf{\widetilde{W}},{\bf
w}}^{\dagger}(\sum_{i\in I,i\neq i_0} w_i^2\pi_{W_{i}}(f)-g)\in
\mathcal{W}^{\perp}$. But then $S_{\mathbf{\widetilde{W}},{\bf
w}}^{\dagger}(\sum_{i\in I,i\neq i_0} w_i^2\pi_{W_{i}}(f)-g)=0$ i.e.
$\sum_{i\in I,i\neq i_0} w_i^2\pi_{W_{i}}(f)-g\in {\rm
Ker}(S_{\mathbf{\widetilde{W}},{\bf w}}^{\dagger}) =
\mathcal{W}^{\perp}$ and so $\sum_{i\in I,i\neq i_0}
w_i^2\pi_{W_{i}}(f)=g$. Then $\sum_{i\in I,i\neq i_0}
w_i^2\pi_{W_{i}}(f)\in \widetilde{W}_{i_0}$. So  $\sum_{i\in I,i\neq
i_0} w_i^2\pi_{W_{i}}(f)\in
(W_{i_0}\cap\overline{\text{span}}\{W_i:i\neq i_0\})^\bot$ and
$\sum_{i\in I,i\neq i_0} w_i^2\pi_{W_{i}}(f)\in W_{i_0}\cap
\overline{\text{span}}\{W_i:i\neq i_0\}$. It follows that
$\sum_{i\in I,i\neq i_0} w_i^2\pi_{W_{i}}(f)=0$. So
$f=S_{\mathbf{\widetilde{W}},{\bf
w}}^{\dagger}S_{\mathbf{\widetilde{W}},{\bf
w}}(f)=S_{\mathbf{\widetilde{W}},{\bf
w}}^{\dagger}w_{i_0}^2\pi_{{\widetilde {W}}_{i_0}}(f)$. Therefore
$f\in S_{\mathbf{\widetilde{W}},{\bf w}}^{\dagger}({\widetilde
{W}}_{i_0})\cap S_{\mathbf{W},{\bf w}}^{\dagger}(W_{i_0}\cap
{\widetilde{W}_{i_0}}^\bot)$, i.e. $f=S_{\mathbf{\widetilde{W}},{\bf
w}}^{\dagger}(h)=S_{\mathbf{W},{\bf w}}^{\dagger}(s)$, where $h\in
{\widetilde {W}}_{i_0}$ and $s\in W_{i_0}\cap
{\widetilde{W}_{i_0}}^\bot$. But then
$S_{\mathbf{\widetilde{W}},{\bf
w}}^{\dagger}(w_{i_0}^2\pi_{{\widetilde
{W}}_{i_0}}\left(\frac{h}{w_{i_0}^2}\right))=S_{\mathbf{W},{\bf
w}}^{\dagger}(w_{i_0}^2\pi_{{W}_{i_0}}\left(\frac{s}{w_{i_0}^2}\right))$,
hence $\frac{h}{w_{i_0}^2}=\frac{s}{w_{i_0}^2}$, and so $h=s=f=0$,
which is a contradiction.
\end{proof}

\section*{Acknowledgement}
The research of S. B. Heineken and P. M. Morillas has been partially
supported by Grants PIP 112-201501-00589-CO (CONICET) and PROIPRO
3-2116 (UNSL). S. B. Heineken also acknowledges the support of
Grants PICT-2011-0436 (UBA) and UBACyT 20020130100422BA.

%\textit{Acknowledgement}: I am very grateful with .

% ----------------------------------------------------------------
%\bibliographystyle{amsplain}

\end{document}